\documentclass{amsart} 
\usepackage{amssymb}
\usepackage{amsmath}
\usepackage{amsfonts}

\newcounter{TmpEnumi}

\begin{document}
\newtheorem{theo}{Theorem}[section]
\newtheorem{prop}[theo]{Proposition}
\newtheorem{lemma}[theo]{Lemma}
\newtheorem{exam}[theo]{Example}
\newtheorem{coro}[theo]{Corollary}
\theoremstyle{definition}
\newtheorem{defi}[theo]{Definition}
\newtheorem{rem}[theo]{Remark}


\newcommand{\Ab}{{\bf A}}
\newcommand{\Bb}{{\bf B}}
\newcommand{\Cb}{{\bf C}}
\newcommand{\Eb}{{\bf E}}
\newcommand{\Fb}{{\bf F}}
\newcommand{\Kb}{{\bf K}}
\newcommand{\Mb}{{\bf M}}
\newcommand{\Nb}{{\bf N}}
\newcommand{\Pb}{{\bf P}}
\newcommand{\Qb}{{\bf Q}}
\newcommand{\Rb}{{\bf R}}
\newcommand{\Sb}{{\bf S}}
\newcommand{\Tb}{{\bf T}}
\newcommand{\Ub}{{\bf U}}
\newcommand{\Vb}{{\bf V}}
\newcommand{\Xb}{{\bf X}}
\newcommand{\Yb}{{\bf Y}}
\newcommand{\Zb}{{\bf Z}}
\newcommand{\Ac}{{\mathcal A}}
\newcommand{\Bc}{{\mathcal B}}
\newcommand{\Cc}{{\mathcal C}}
\newcommand{\Dc}{{\mathcal D}}
\newcommand{\Fc}{{\mathcal F}}
\newcommand{\Ic}{{\mathcal I}}
\newcommand{\Jc}{{\mathcal J}}
\newcommand{\Lc}{{\mathcal L}}
\newcommand{\MM}{{\mathcal M}}
\newcommand{\Oc}{{\mathcal O}}
\newcommand{\Pc}{{\mathcal P}}
\newcommand{\Sc}{{\mathcal S}}
\newcommand{\Tc}{{\mathcal T}}
\newcommand{\Uc}{{\mathcal U}}
\newcommand{\Vc}{{\mathcal V}}

\newcommand{\db}{{\bf d}}
\newcommand{\fb}{{\bf f}}
\newcommand{\gb}{{\bf g}}

\newcommand{\ax}{{\rm ax}}
\newcommand{\Acc}{{\rm Acc}}
\newcommand{\Act}{{\rm Act}}
\newcommand{\ded}{{\rm ded}}
\newcommand{\Gm}{{$\Gamma_0$}}
\newcommand{\ID}{{${\rm ID}_1^i(\Oc)$}}
\newcommand{\PAP}{{${\rm PA}(P)$}}
\newcommand{\ACA}{{${\rm ACA}^i$}}
\newcommand{\RefP}{{${\rm Ref}^*({\rm PA}(P))$}}
\newcommand{\RefS}{{${\rm Ref}^*({\rm S}(P))$}}
\newcommand{\Rfn}{{\rm Rfn}}
\newcommand{\tar}{{\rm Tarski}}
\newcommand{\UNFA}{{${\mathcal U}({\rm NFA})$}}

\author{Nik Weaver}

\title [Axiomatizing mathematical conceptualism]
       {Axiomatizing mathematical conceptualism in third order arithmetic}

\address {Department of Mathematics\\
          Washington University in Saint Louis\\
          Saint Louis, MO 63130}

\email {nweaver@math.wustl.edu}

\date{\em May 23, 2007}

\begin{abstract}
We review the philosophical framework of mathematical conceptualism
as an alternative to set-theoretic foundations and show how mainstream
mathematics can be developed on this basis. The paper includes an
explicit axiomatization of the basic principles of conceptualism
in a formal system CM set in the language of third order arithmetic.
\end{abstract}

\maketitle


This paper is part of a project whose goal is to make a case
that mathematics should be disassociated from set theory.
The reasons for wanting to do this, which I discuss in greater
detail elsewhere (\cite{W4}; see also \cite{W1} and \cite{W5}),
involve both the philosophical unsoundness of set theory
and its practical irrelevance to mainstream mathematics.

Set theory is based on the reification of a collection as a
{\it separate object}, an elementary philosophical
error. Not only is this error obvious, it also has
the spectacular consequence of immediately giving rise to the
classical set theoretic paradoxes. Of course, these paradoxes
are not derivable in the standard axiomatizations of set theory,
but that is only because these systems were specifically designed
to avoid them. In these systems
the paradoxes are blocked by means of ad hoc
restrictions on the set concept that have no obvious intuitive
justification, which has led to the development of a large literature of
attempted rationalizations (e.g., \cite{Boo, Chi, Fie, Jan, Lew, Mad1,
Mad2, Par, Sho}). The heterogeneity of these efforts attests to the
difficulty of this task. For example, from a platonistic perspective
it seems impossible to give a cogent, principled explanation of
why it should be legal to form power sets of infinite sets, given that
unrestricted comprehension (forming the set of all $x$ such that
$P(x)$) is not supposed to be
valid in general. Antiplatonistic attempts to justify
set theory, on the other hand, appear doomed from the start because of
the massive gap in consistency strength
between straightforwardly antiplatonistically justifiable
systems like Peano arithmetic and, say, Zermelo-Frankel set theory.
The fact that modern mathematics apparently rests on this kind of basis
must be considered a major embarassment for the subject.

Probably the real appeal of set theory comes not from any murky
philosophical defense, but rather from the role it plays as the
standard foundation for mathematics. However, its concordance
with normal mathematical practice is actually quite poor. Cantorian
set theory postulates a vast universe of sets containing
remote cardinals which bear no relation to the relatively
concrete world of ordinary mathematics, where most objects of central
interest are essentially countable (i.e., separable for some natural
topology). Similarly,
set theory as a mathematical discipline is quite isolated from the
rest of mathematics, and it could hardly be otherwise given the gap
between its subject matter and the subject matter of normal mathematics.

One might still claim that even if set theory does not fit
mainstream mathematics very well, it nonetheless does so better than
any foundational alternative. The main purpose of the present paper is
to show that this is false, by explaining how ordinary mathematics can
be developed in a concrete way that avoids the metaphysical extravagance
and nonseparable pathology of Cantorian set theory. The general
point is not new: many authors have observed that large amounts of
mainstream mathematics can be developed in surprisingly weak systems.
I have already done something like this myself in \cite{W2}. (Also
see the introduction to \cite{W2} for other references, and particularly
see \cite{Sim}, which contains a very thorough development that is relatively
close to what we do here.) Actually, for a classically trained reader
\cite{W2} may be easier to read than the present paper, because the
approach taken there was modeled on the usual set-theoretic development
of mathematics (in particular, the formal language used there was the
usual language of set theory); here the goal is not to mimic set-theoretic
mathematical foundations, but rather to find an approach derived
more directly from an alternative philosophical basis.

The philosophical approach we adopt, mathematical conceptualism, is a
refinement of the {\it predicativist} philosophy of Poincar\'e and Russell.
The basic idea is that we accept as legitimate only those structures that
can be constructed, but we allow constructions of transfinite length.
What makes this ``conceptual'' \cite{W1, W5} is that we are concerned
not only with those constructions that we can actually physically carry
out, but more broadly with all those that are conceivable (perhaps
supposing our universe had different properties than it does). The admission
of transfinite processes takes us well beyond intuitionism, which only
allows finite constructions, but at the same time our insistence on
having some degree of constructivity is far more restrictive than
full-blooded platonism. The result is a foundational stance that
matches actual mathematical practice much better than either of
these two alternatives.

At the level of countable structures there is little practical difference
between conceptualism and platonism. However, uncountable structures in
general can be only partially realized in the conceptualist framework.
For example, although we can (transfinitely) construct individual real numbers
(regarded, say, as Dedekind cuts), we have no clear picture of a
transfinite construction that would succeed in producing the entire real
line. Thus we might say that the real line exists conceptualistically
only in an unfinished state.

(A platonist might counter that the well-ordering theorem does give him
a picture of how the real line could be sequentially constructed one
element at a time. This is not a good argument because it is the set
theoretic axioms --- specifically, the power set axiom --- on the
basis of which the well-ordering theorem is proven that are in question
here. If we take conceivability as a first principle then uncountable
constructions in general become highly dubious; see Section 1.2.)

We regard the idea of a completed surveyable
real line in roughly the same way that we regard naive infinitesimals, as
an evocative idealization that does not really have a definite meaning.
Admittedly, this is at odds with normal mathematics, which does treat
the real line as a completed and in some sense surveyable structure.
However, it does not seem
that this assumption is actually used in any serious way in
mainstream mathematics. The standard developments of all
mainstream subjects can be executed perfectly well in a conceptualistic
setting which treats the real line and other structures at a similar
level of complexity as only incompletely realizable.

The unfinished nature of the real line introduces logical subtleties
which we handle by adopting intuitionistic logic when quantifying over
all real numbers. This is one of the principal differences between
conceptualism and earlier versions of predicativism, where classical
logic was used almost exclusively and there was persistent confusion
about the legitimacy of second order quantification (e.g., \cite{Fef}).
Alternatively, one could avoid the use of intuitionistic logic by
arresting the construction of the mathematical universe at some
natural point and reasoning classically about the resulting
fixed partial universe; this was the approach adopted in \cite{W2}.

In the present paper we also go further and allow some reasoning
about arbitrary sets of real numbers, although this requires even
greater care. This represents a change from the point of view expressed
in Section 2.5 of \cite{W3}, where I would have rejected any reference
to arbitrary sets of real numbers. (However, I stand on the main
point of that discussion, that self-applicative schematic predicates are
prima facie predicatively invalid.) I now believe that reasoning at this
level of abstraction may be legitimate provided an even weaker logical
apparatus, the minimal logic of Johansson, is used. The justification
for this conclusion is explained in \cite{W5}. It is perhaps not
crucial here because we are going to work with a restricted notion
of sets of real numbers to which ordinary intuitionistic logic does
apply.

We will present a formal system CM for conceptualist mathematics
and outline how core mathematics can be developed within this system.
The claim is that virtually all mainstream mathematics can be
straightforwardly realized in CM.

Our system CM is similar to systems in \cite{Sim}, and there is a
strong resemblance between our development and that in \cite{Sim}
(though the similarity to \cite{W2} seems greater).
The main differences lie in our use of third order
variables, which simplifies the presentation in some ways, and our use
of non-classical logic, which is an important theoretical distinction
but has surprisingly little practical effect. Most assertions of
interest in mainstream subjects can be reduced to questions
involving quantification only over countable sets, at which point
classical logic can be used.

\section{Philosophical motivation}

\subsection{The concept of a set}
Sets are typically defined as ``collections of objects'', and,
crucially, these collections are themselves supposed to be objects
capable of belonging to other sets. This seems to be a simple grammatical
confusion. If we can talk sensibly about the set of all books in
the Library of Congress, then ``the set of all books in the Library
of Congress'' must be a particular {\it thing}, the reasoning apparently
goes; it is not a physical object, so it must be a non-physical
object. But we can also talk sensibly about the average taxpayer;
should we infer that this is an actual (albeit non-physical) person?
Is it really  coherent to maintain that {\it the set of all taxpayers}
is a genuine ``abstract object'' while conceding that {\it the average
taxpayer} is just a figure of speech?

Despite its nonsensical official rationale, in sufficiently concrete
settings the apparatus of set theory can be straightforwardly justified.
For example, if we want to talk about sets of natural numbers as if they
were actual objects, we can refer instead to infinite sequences of 0's
and 1's. Philosophical questions can still be raised about the general
concept of an infinite sequence of 0's and 1's, but these are
questions about the notion of infinity which could be brought
against any interpretation of mathematics. The point is that
we are no longer postulating the existence of fictional entities
based on a grammatical confusion.

In other words, set theory can be legitimized to the extent that we
are able to set up a system of {\it token entities} which can play the
role of sets.  And it is clear that most of our intuition about sets
comes from such token structures. When I think of a set, it is not
the set itself that I picture --- how could I, when this is supposed
to be an abstract object that has no visual aspect --- but rather
some structural representation. Perhaps this explains G\"odel's famous
comment that we have ``something like a perception'' of sets
(\cite{God}, p.\ 484). We do not have any perception of sets, but
we do have a conception, perhaps something like a perception, of
structures which can play the role of sets.

The philosophical literature on set theory tends to blur this distinction.
In particular, it is highly ambivalent as to whether sets can actually be
{\it formed} or {\it manipulated}
in any sense, or whether they are really supposed to be
absolutely inert abstractions. The latter is the official platonist position,
but language suggesting the former is ubiquitous. This is particularly
seen in the most popular explanation of the paradoxes of naive set theory,
which invokes an ``iterative conception'' of sets according to which sets
are to be thought of as being iteratively constructed in stages.
Of course, this makes no sense if they are simultaneously thought of as
being inert abstractions.

(Most authors who write about the iterative conception are clearly aware
of this difficulty. This sometimes results in strange comments to the effect
that the informal explanation of the iterative conception in quasi-constructive
terms is not supposed to be taken
literally, which obviously begs the questions of how, then, one is
supposed to take it, and what meaning the iterative conception can have if
it cannot be explained in a way that makes literal sense.)

In practice, mathematical reasoning about sets consists largely of
imagined quasi-physical manipulations that could hardly apply to
causally inert metaphysical entities. We order sets, we remove elements
from them, we form products, we cut and paste.
Practically anything we would call a
mathematical construction makes sense only as applied to quasi-physical
structures, not abstract sets. This suggests that a rational approach to
set theory would drop the nonsensical ``abstract objects'' interpretation
of sets and focus entirely on possible structures \cite{Hel}. But we then
have to ask to what extent traditional set theory can be justified in this
approach. In particular, it would not seem to justify set-theoretic
``constructions'' like the power set operation which do not correspond
to any obvious quasi-physical construction.

\subsection{Uncountable structures}

Our choice of tokens for arbitrary sets of natural numbers is
special to that case, but it is natural to assume that a reasonable
system of tokens for arbitrary sets of X's could be found for any
choice of X. However, that intuition may be misleading.
There is no obvious system of tokens which could be used to
represent arbitrary sets of infinite sequences of 0's and 1's,
for example.

We have to specify more clearly what would count as a token. We
need not insist that these actually be physically
realizable in our universe. Above I used the term ``quasi-physical''
with the intention of suggesting something like ``physically
realizable in some conceivable universe''. This seems like the
right criterion to use given that we want to imagine manipulating
these tokens in the ways mentioned above. So: could a meaningful
system of tokens for arbitrary sets of infinite sequences of 0's
and 1's appear in any conceivable universe?

The difficulty here is that in this example the tokens themselves would
presumably have to be uncountably large. So the question becomes whether
it is possible to clearly conceive of a universe that contains uncountable
structures.

It is not so hard to imagine living in a universe that is infinite
in extent and contains a countable infinity of physical objects.
Indeed, it is quite conceivable that our own universe has this property.
One can even imagine a universe that is finite in extent but still
contains a countable infinity of non-overlapping physical objects
whose sizes decrease rapidly enough that they collectively fill
only a finite volume. But what would it be like to live in a
universe containing {\it uncountably many} non-overlapping physical
objects?

Such a universe is difficult to imagine, and there is a good reason
for this: the L\"owenheim-Skolem theorem. According to this theorem,
any formal description one could give of an uncountable universe
would be equally true of some countable substructure. Thus, there
is nothing we can say, at least formally, that would serve to
distinguish an uncountable universe from a countable one. This
suggests that we have no a priori conception of any uncountable
universe, which presents a serious obstruction to any attempt at
setting up a system of tokens to model a classically uncountable structure.

Taking conceivability as a first principle really calls the whole
notion of uncountability into question. All classical methods
of constructing uncountable sets ultimately rely on the power
set operation. But from the conceptualist perspective the
latter cannot be justified unless one has a prior concept of
uncountable structures. This point may puzzle mathematicians who
are accustomed to thinking of the power set operation as just one
among many straightforward tools for constructing sets. The crucial
distinction that is missed here is between collections that are
{\it surveyable} (we have a clear picture of what it would mean to
exhaustively search such a collection) and those that are merely
{\it determinate} (we can decide whether any given object belongs
to the collection, but have no clear idea even in principle how one
would go about searching through all objects in the collection).
The natural numbers are surveyable, their power ``set'' is merely
determinate.

We have a clear idea of how we would go about searching through the
natural numbers, but not how we could search through the real numbers.
So if ``set'' means ``surveyable collection'' then we cannot a priori
assume that the real line is a set. But if ``set'' means ``determinate
collection'' then basic set-theoretic constructions become problematic:
for instance, it is not clear that we can take the union of a family of
sets indexed by a determinate set since the result might not be
determinate. We cannot decide whether any given object belongs to
the union because this might require searching through the entire
index set.

Somehow the vague sense that the natural numbers and the real line are
both ``sets'' has led us to transfer intuition about one to the other,
resulting in the idea that the power set operation is just a
straightforward set construction. It is not.

Since uncountability is embedded so firmly in current thought,
this point must be emphasized: {\it we have no obvious justification
for the existence of uncountable structures that does not rely
on the conflation of surveyable and determinate collections}.
Moreover, the L\"owenheim-Skolem theorem gives us a powerful reason
to expect that no conceptualistic justification of uncountable
structures is possible. If we cannot describe what it would be like
to exist in a universe containing uncountable structures (in a way
that would distinguish it from some countable subuniverse), then we
presumably cannot imagine it either.

Thus, we reject the notion of actually existing uncountable
structures, not just in our universe but in any conceivable
universe.

\subsection{Countable constructions}

As I mentioned in the last section, I think it is quite possible that
our universe is infinite in extent and contains infinitely many disjoint
physical objects. Regardless of whether this is actually the case, it
would be difficult to argue that this possibility is literally
inconceivable. Therefore, we can accept that countably infinite
structures are legitimately part of the conceptualist landscape.

A more subtle question involves the conceivability of transfinite
computations or constructions. By this I mean processes that not
only potentially involve any finite number of steps, but literally
involve infinitely many steps, and may even continue on after
infinitely many steps have been completed. The philosophical term
for such a process is ``supertask''. For example, we can imagine
resolving the Goldbach conjecture by mechanically searching through
the natural numbers for a counterexample; if no counterexample is
found, the conjecture must be true. This example could require the literal
execution of infinitely many steps, since we might not get an answer
until after all numbers have been checked.

But can we really imagine carrying out and completing such a
process? Can we imagine {\it what it would be like} to live in
a universe in which such supertasks were possible? If not, then
they cannot be admitted into the conceptualist picture.

It may not be obvious that literally infinite constructions
are really conceivable. However, several recent proposals for
carrying out supertasks make it quite plausible that one can
indeed form a perfectly coherent picture of what doing this
would be like. My favorite is a suggestion due to Davies
\cite{Dav} that he calls ``building infinite machines''.

Briefly, Davies' idea is something like this. Suppose we want to
write down all the natural numbers. To achieve this, we build a
machine that accepts as input a natural number $n$, and on this
input it first writes down the number $n$ (say, in decimal notation)
in front of itself; then, to its right, it builds a copy of itself half
its size that runs twice as fast; and finally it feeds this smaller copy
the input $n+1$ and sets it running. Having built such a machine, we feed
it the number 1 as input and set it running. It writes down the number
1, builds a smaller and faster copy of itself, and feeds it the
number 2 as input; that smaller copy writes down the number 2,
builds a still smaller and faster machine, and feeds it the number
3 as input; and so on. Because the machines are speeding up
exponentially the entire task is completed in a finite amount of
time. Because they are shrinking exponentially it takes place in
a finite spatial region.

Davies observes that experiments like this would be possible in a
continuous Newtonian universe. I am not so sure that something like
this is not actually possible in our universe (with the machines being
being ``built'' as perturbations of the electromagnetic field, say),
but for our purposes here merely being able to imagine a universe in which
it could take place is sufficient. I see no logical obstruction to a
universe in which an experiment of Davies' type could actually be performed.

The task I have just described is a particularly simple kind of
supertask; although it does require the literal completion of infinitely
many steps, its goal is merely to write something down. More
problematic would be a supertask like the one which checks the
truth of the Goldbach conjecture and returns an answer. It is
not hard to imagine programming a Davies machine to search
through the natural numbers looking for a counterexample, but
since the size of the machines carrying out this task goes to
zero and we want a final answer to be returned on the scale of
the first machine, this seems to require a discontinuity. In a
universe whose time evolution is continuous there may well be
a logical obstruction to carrying out such a task.

On the other hand, are universes with discontinuous time
evolution literally inconceivable? This would come as a surprise
to proponents of the Copenhagen interpretation of quantum mechanics.
In this interpretation the time evolution in our universe is thought
to involve a discontinuous ``collapse'' of the state vector whenever
an observer makes a measurement. This view can be criticized in
many ways, but one criticism I have never seen is that it is
inconceivable because it involves a discontinuous time evolution.
Indeed, it is easy to imagine universes whose time evolution is
discontinuous.

If discontinuities are possible, then we can easily imagine building
a sequence of Davies machines next to a wire, say, and giving each
machine the ability to send a signal along the wire if it finds the
counterexample it is looking for. A discontinuity appears in the
assumption that we have the ability to detect a signal coming from any
of the machines, no matter how small. On the basis of thought experiments
like these, we can justify the conceivability not only of countably
infinite structures, but also of some transfinite computations or
constructions.

\subsection{Quantifying over the reals}

We have seen that the power set of the natural numbers does not
exist conceptualistically as a well-defined structure. But we have
also seen that the general notion of a set of natural numbers,
realized as an infinite sequence of 0's and 1's, makes perfect sense
conceptualistically. The point is that even though we have
no notion of a construction that would generate {\it all} sequences
of 0's and 1's, we can nonetheless recognize such a sequence when
we see one (or at any rate we could build a Davies machine to perform
such a check).

Similarly, regarding real numbers as Dedekind cuts, our notion of a
general real number is in the preceding sense completely definite, but the
entire real line does not exist conceptualistically as a well-defined
structure.

Because we do understand the general concept of a real number, we can regard
assertions which quantify over all real numbers as intelligible, provided
we are slightly careful about interpreting the meaning of the quantifiers.
For instance, the classical interpretation of ``there exists'' only makes
sense when quantifying over the objects appearing in some
well-defined structure, which we do not
have in this case. However, we can still interpret ``there exists''
constructively, i.e., we can understand an assertion of the existence
of a real number with some property to be an assertion that we have some
way of constructing it. Similarly, we take an assertion that all real
numbers have some property not as meaning that if one checked all real
numbers one would find they all have this property --- this cannot be
done, even in principle --- but rather as meaning that we can know in
advance that any real number that appears will have the property in
question.

This is just the usual intuitionistic interpretation of quantifiers.
Indeed, intuitionistic logic is exactly suited to the present
situation. I will review this logic in Section 2.1. Its most salient
property is that the law $A \vee \neg A$ is not generally valid.
Classically this law can be justified by saying that the truth value
of any statement could in principle be determined by a mechanical
search (much like the Goldbach conjecture in our discussion above);
thus any statement is definitely either true or false. However, when
we quantify over real numbers this justification is not available.
Since the notion of ``all real numbers'' is indefinite in the sense that
there is no structural amalgamation of all real numbers which could
be mechanically surveyed, even in principle, we cannot generally
affirm that any statement about arbitrary real numbers has a definite
truth value. Or, at least, since there could be statements whose truth
value cannot be determined even in principle, such an affirmation would
not have any substantive content.

The main point to keep in mind is that we cannot assume that all
statements have definite truth values; to a large extent,
intuitionistic logic merely codifies the forms of reasoning that
one would naturally adopt in such a case.

\subsection{Quantifying over sets of reals}

Above I have distinguished between fixed, surveyable structures
which can be conceived in toto and determinate concepts
which can be partially, but never wholly, realized by actual concrete
structures. For us classically countable structures generally fall in the
first category and classical structures at the same level of complexity
as the real line fall in the second. I have also claimed that
classical logic is appropriate only when quantifying over surveyable
structures and that intuitionistic logic should be used when quantifying
over indefinitely extendable but determinate concepts.

It must be emphasized that the individual elements that make up an
indefinitely extendable concept such as the real line are themselves
concretely realizable structures. The logical subtlety arises when
we discuss arbitrary structures of some type, e.g., arbitrary sets
of natural numbers.

We now want to consider the classical concept of an arbitrary set of
real numbers. Here even an individual --- a single set of real numbers
--- in general can never be concretely realized. So arbitrary sets of
real numbers can have concrete reality only as something like rules
or prescriptions telling us how to decide, as new real numbers become
available, whether they should be accepted into or rejected from the
set.

I argue in \cite{W5} that the correct logic to use when reasoning
at this level of abstraction is Johansson's minimal logic. This is
a weakening of intuitionistic logic which omits the ex falso law
that states that any assertion follows from a contradiction. The
idea is that, unlike concrete structures, rules or prescriptions
have a genuine potential to be contradictory. Thus, we can set up
standards for reasoning about arbitrary rules, which is what minimal
logic does, but we are not allowed to decree that such rules will be
free from contradiction. The subtle point is that the meaning of the
rules under discussion depends on the logical apparatus used to reason
about them, so incorporating an assertion of consistency into that
apparatus would be circular.

This point is pursued further in \cite{W5}. It is somewhat peripheral
here since we will work in a restricted setting where the ex falso law
is valid. (We consider only ``determinate'' subcollections of $\Pc(\omega)$
for which the membership relation is decidable, rather than the ``definite''
subcollections discussed in \cite{W5}. Thus the law of excluded middle
holds for all atomic formulas, which justifies ex falso; see Section 2.3
below.)

\subsection{Summary}

We reject the proposition that sets literally exist as some kind of
ghostly non-physical objects. This assertion has no meaningful content,
it is responsible for the classical paradoxes of naive set theory, and it is
not even compatible with the grammar of set language in ordinary speech
\cite{Sla}.

The way mathematicians use set language is also incompatible with
the concept of a set as an inert unitary object existing in some
timeless metaphysical realm. Mathematicians handle sets as if they
were articulated structures capable of being physically manipulated.
For the purposes of understanding the foundations of mathematics, it
therefore makes sense to focus on the concept of {\it a possible
structure in a conceivable universe}.

This radically changes the nature of mathematical foundations, most
importantly by delegitimizing the power set operation. There is no
obvious ``construction'' of a structure whose elements correspond to
all substructures of a given structure. Moreover, a general argument
can be made that this would actually be incompatible with the notion
of a conceivable universe. Any conceivable universe should be finitely
describable, and hence subject to the L\"owenheim-Skolem theorem, which
would entail the incompleteness of any putative infinite ``power set''
structure.

Indeed, the applicability of the L\"owenheim-Skolem theorem to conceivable
universes casts doubt on the general concept of an actually uncountable
structure. However, a good case can still be made for the literal
conceivability of transfinite constructions of length $\omega$,
$\omega^2$, $\omega^\omega$, etc.

In the foundational picture that emerges countable structures are
seen as legitimate but we do not accept the idea of an existing
completed uncountable structure. Rather, we treat concepts such
as ``real number'' or ``set of natural numbers'' as definite
concepts which can be only incompletely realized as actual structures.
This calls for the use of intuitionistic logic when we quantify over
all real numbers or all sets of natural numbers.

We can also think of, for example, sets of real numbers as prescriptions
telling us how to decide whether each new real number, as it becomes
available, is to be accepted or rejected. It does not seem possible to
meaningfully iterate this process any further. Since ``legitimate
prescription'' is not even a sharp concept, it is hard to imagine how
one could sensibly model the concept of an arbitrary set of such things.

Our task is now to set out a formal system which expresses the
philosophical point of view set out above, and to show that it
naturally accomodates the vast bulk of normal mainstream mathematics.

\section{The system CM}

\subsection{Systems of logic}

Intuitionistic logic is most easily understood in terms of systems of
``natural deduction'' \cite{Pra}. For more details the reader may also
consult the excellent survey article \cite{vD}.

Informally, the rules for natural deduction in minimal logic are:
\begin{itemize}
\item
Given $\phi$ and $\psi$ deduce $\phi \wedge \psi$;
given $\phi \wedge \psi$ deduce $\phi$ and $\psi$.
\item
Given either $\phi$ or $\psi$ deduce $\phi \vee \psi$;
given $\phi \vee \psi$, a proof of $\sigma$ from $\phi$, and a
proof of $\sigma$ from $\psi$, deduce $\sigma$.
\item
Given a proof of $\psi$ from $\phi$ deduce $\phi \to \psi$;
given $\phi$ and $\phi \to \psi$ deduce $\psi$.
\item
Given $\phi(x)$ deduce $(\forall x)\phi(x)$; if the term
$t$ is free for $x$, given $(\forall x)\phi(x)$ deduce $\phi(t)$.
\item
If the term $t$ is free for $x$, given $\phi(t)$ deduce
$(\exists x)\phi(x)$; if $y$ does not occur freely in $\psi$, given
$(\exists x)\phi(x)$ and a proof of $\psi$ from $\phi(y)$ deduce $\psi$.
\end{itemize}

Natural deduction involves the
concept of ``a proof of $\psi$ from $\phi$''; thus, for example, we may
temporarily assume $\phi$, use this to prove $\psi$, and then infer the
statement $\phi \to \psi$. Temporary assumptions can be nested, just as
in normal informal reasoning, and one has to be slightly careful
about allowing the variables involved in the quantifier rules to appear
in active temporary assumptions. See \cite{Pra} or \cite{vD} for a
precise exposition.

Despite this minor complication, it should be evident that the rules
of natural deduction are indeed very natural and correspond exactly
to ordinary informal reasoning. It could be said that these rules
are nothing more than a direct expression of the meaning of the
various logical sybols.

We let $\perp$ stand for the statement ``$0 = 1$'' and define
negation by letting $\neg \phi$ stand for $\phi \to \perp$. Minimal logic
contains no special rules for negation; in intuitionistic logic
we adopt the ex falso rule ``given $\perp$ deduce $\phi$'' for any
formula $\phi$, and in classical logic we also adopt the law of
excluded middle ``$\phi \vee \neg\phi$'' for any formula $\phi$.
As I explained above in Sections
1.4 and 1.5, this is justified when we are reasoning about objects
belonging to a fixed surveyable structure, but not in broader settings.

It is easy to learn to reason intuitionistically using
natural deduction. For example, we prove $(\phi \to \psi) \to
(\neg\psi \to \neg\phi)$ as follows: (1) assume $\phi \to \psi$;
(2) assume $\neg \psi$; (3) assume $\phi$; (4) deduce $\psi$ from (1)
and (3); (5) deduce $\perp$ from (2) and (4); (6) cancel assumption (3)
and infer $\neg\phi$ from (5); (7) cancel assumption (2) and infer
$\neg\psi \to \neg\phi$ from (6); (8) cancel assumption (1) and
infer $(\phi \to \psi) \to (\neg\psi \to \neg\phi)$ from (7).
Exercise: verify
$\neg(\phi \vee \psi) \leftrightarrow (\neg\phi \wedge \neg\psi)$
and $(\neg\phi \vee \neg\psi) \to \neg(\phi \wedge \psi)$ (both of these
can be proven in minimal logic), and convince oneself that the last
implication cannot be reversed without using excluded middle. For practice
with quantifiers the reader can check that $(\forall x)\neg\phi(x)
\leftrightarrow \neg(\exists x)\phi(x)$ and $(\exists x)\neg\phi(x) \to
\neg(\forall x)\phi(x)$ (again, the last implication cannot be reversed
without using excluded middle). A simple example of a basic law whose proof
requires the use of ex falso is $[(\phi \vee \psi) \wedge \neg \phi]
\to \psi$.

Minimal logic can also be formulated in the following
more standard way. We now eliminate the falsehood symbol $\perp$ and
take negation as primitive. There are three logical rules of inference:
\begin{enumerate}
\item
given $\phi$ and $\phi \to \psi$ deduce $\psi$
\item
given $\phi \to \psi$ deduce $\phi \to (\forall x)\psi$
\item
if $x$ does not occur freely in $\psi$, given $\phi \to \psi$
deduce $(\exists x)\phi \to \psi$.
\end{enumerate}
There are eleven logical axiom schemes:
\begin{enumerate}
\item
$\phi \to (\psi \to \phi)$
\item
$(\phi \to \psi) \to [(\phi \to (\psi \to \sigma)) \to (\phi \to \sigma)]$
\item
$\phi \to [\psi \to (\phi \wedge \psi)]$
\item
$\phi \wedge \psi \to \phi$
\item
$\phi \wedge \psi \to \psi$
\item
$\phi \to (\phi \vee \psi)$
\item
$\psi \to (\phi \vee \psi)$
\item
$(\phi \to \sigma) \to [(\psi \to \sigma) \to (\phi \vee \psi
\to \sigma)]$
\item
$(\phi \to \psi) \to [(\phi \to \neg\psi) \to \neg\phi]$
\item
$\phi(t) \to (\exists x)\phi(x)$
\item
$(\forall x)\phi(x) \to \phi(t)$
\setcounter{TmpEnumi}{\value{enumi}}
\end{enumerate}
In axioms 10 and 11, $t$ is a term which is free for $x$ in $\phi$.

Intuitionistic logic includes the axioms
\begin{enumerate}
\setcounter{enumi}{\value{TmpEnumi}}
\item
$\phi \to (\neg \phi \to \psi)$
\setcounter{TmpEnumi}{\value{enumi}}
\end{enumerate}
and classical logic additionally includes the axioms
\begin{enumerate}
\setcounter{enumi}{\value{TmpEnumi}}
\item
$\phi \vee \neg\phi$.
\end{enumerate}

\subsection{The formal system CM}

We introduce the formal language of third order arithmetic. There
are three kinds of variables: first order variables $a, b, c, \ldots$
(thought of as ranging over $\omega$), second order variables
$A, B, C, \ldots$ (thought of as ranging over subsets of $\omega$), and
third order variables $\Ab, \Bb, \Cb, \ldots$ (thought of as ranging
over subsets of $\Pc(\omega)$). More concretely, we think of first order
variables as representing the objects appearing in some
infinite sequence, second order variables as
representing infinite sequences of 0's and 1's, and third order
variables as representing predicates which take any second order object
as input and return either 0 or 1. A helpful heuristic is: first order
objects are urelements, second order objects are sets of first order
objects, and third order objects are classes of second order objects.

Numerical terms are built up from the number variables and the constant
symbol $0$ using the successor operation $'$ and the binary operations
$+$ and $\cdot$. The atomic formulas of the language consist of all
formulas of the form $t_1 = t_2$, $t_1 \in X$, and $X \in \Xb$ where
$t_1$ and $t_2$ are numerical terms, $X$ is a second order variable, and
$\Xb$ is a third order variable. General formulas are built up from
the atomic formulas using the logical connectives $\wedge, \vee,
\neg, \to$ and the quantifiers $\forall n$, $\exists n$,
$\forall X$, $\exists X$, $\forall \Xb$, and $\exists \Xb$ for any
first, second, and third order variables $n$, $X$, and $\Xb$. Other
symbols such as $\leftrightarrow$, $\subseteq$, etc., are defined
in terms of the above symbols in the usual way, except for $\subseteq$
at the third order level, which we will define in Definition 3.3.

The logical apparatus of CM is intuitionistic logic, as described in Section
2.1. Additionally, we adopt the law of excluded middle $\phi \vee \neg \phi$
for all formulas $\phi$ that contain no second or third order quantifiers.
We also include the usual axioms for equality between terms.

We now state the non-logical axioms of CM. Throughout the
following $m$ and $n$ are any first order variables, $X$, $Y$, and $Z$
are any second order variables, and $\Xb$ is any
third order variable.
\smallskip

\noindent I. {\it Number axioms}:
\begin{enumerate}
\item
$\neg (n' = 0)$
\item
$m' = n' \to m = n$
\item
$m+0 = m$
\item
$m+n' = (m + n)'$
\item
$m\cdot 0 = 0$
\item
$m\cdot n' = (m\cdot n) + m$
\setcounter{TmpEnumi}{\value{enumi}}
\end{enumerate}

\noindent II. {\it Induction and recursion axioms}:
\begin{enumerate}
\setcounter{enumi}{\value{TmpEnumi}}
\item
$[\phi(0) \wedge (\forall n)(\phi(n) \to \phi(n+1))]
\to (\forall n)\phi(n)$
\item
$(\forall n)(\forall X)(\exists Y)\phi(n,X,Y) \to
(\forall X)(\exists Z)[Z_{(0)} = X \wedge
(\forall n)\phi(n, Z_{(n)}, Z_{(n')})]$
\setcounter{TmpEnumi}{\value{enumi}}
\end{enumerate}
for all formulas $\phi$.
\smallskip

\noindent III. {\it Comprehension axioms}:
\begin{enumerate}
\setcounter{enumi}{\value{TmpEnumi}}
\item
$(\forall n)(\phi(n) \vee \neg \phi(n)) \to
(\exists X)(\forall n)(n \in X \leftrightarrow \phi(n))$
\item
$(\forall X)(\phi(X) \vee \neg \phi(X)) \to
(\exists \Xb)(\forall X)(X \in \Xb \leftrightarrow \phi(X))$
\end{enumerate}
for all formulas $\phi$ containing no free occurrences of $X$ in (9),
and no free occurrences of $\Xb$ in (10).
\smallskip

The notation $Z_{(n)}$ is defined below in Definition 3.2.
Essentially, axiom (8) states that if for all $n$ and $X$ there exists
$Y$ such that $\phi(n,X,Y)$, then for any $X$ we can find a sequence
$(Z_n)$ such that $Z_0 = X$ and $\phi(n,Z_n, Z_{n+1})$ holds for all $n$.
\smallskip

\subsection{Discussion}

We adopt the usual intuitionistic interpretation of the logical
symbols, e.g., $(\exists n)\phi(n)$ means that we have (in principle)
a way to find a value of $n$ satisfying $\phi$; $\phi \to \psi$
means that we can convert any proof of $\phi$ into a
proof of $\psi$; and so on.

The rules of minimal logic presented in Section 2.1 directly express
the meanings of the logical symbols, and all atomic formulas satisfy
the law of excluded middle. (We have excluded middle for the statement
$n \in X$ for any $n$ and $X$ since any sequence of 0's and 1's could
in principle be surveyed to determine its truth value, and we have
excluded middle for the statement $X \in \Xb$ since third order objects
are assumed to assign a definite truth value to this statement for any
$X$.) This is enough to justify the ex falso law; to see this,
for each atomic formula $A = A(x_1, \ldots, x_n)$ introduce a
function symbol $f_A$ together with the axiom
$$(A(x_1, \ldots, x_n) \leftrightarrow f_A(x_1, \ldots, x_n) = 1) \wedge
(\neg A(x_1, \ldots, x_n) \leftrightarrow f_A(x_1, \ldots, x_n) = 0).$$
This extension of the original system should be unproblematic. But now
assuming $\bot$ we can deduce $A \to \bot$, i.e., $\neg A$, for any
atomic formula $A = A(x_1, \ldots, x_n)$, then infer $f_A(x_1, \ldots, x_n)
= 0$, then (since $0 = 1$) infer $f_A(x_1, \ldots, x_n) = 1$, and finally
infer $A$. So in the extended system we can actually deduce any atomic
formula $A$ from $\bot$, which is enough to verify ex falso for the
original system.

(In \cite{W5} we considered a broader notion of third order objects,
the ``definite'' subsets of $\Pc(\omega)$. These do not have a decidable
membership relation, so the above justification of the ex falso law
would not be valid. On the other hand, this choice of third order
objects would support a stronger comprehension scheme according to
which $(\exists \Xb)(\forall X)(X \in \Xb \leftrightarrow \phi(X))$
for any formula $\phi$ which contains no third order quantifiers
and no free occurrence of $\Xb$.)

The truth value of any formula with no second or third order quantifiers
could be determined by a countable computation, so we adopt
the law of excluded middle for such formulas.

The number axioms (1) -- (6) assert basic facts which are evidently true
for any $\omega$-sequence. The induction axioms (7) reflect
our acceptance of countable procedures, as they could
be verified by deductions of length $\omega$.

The recursion axioms (8) are usually called ``dependent choice'' \cite{Sim},
but in the context of intuitionistic logic they should not be understood as
choice axioms in the traditional sense.
Intuitionistically we interpret the assertion $(\forall x)\phi(x)$
as expressing that there is a {\it uniform} proof of $\phi(x)$ for all $x$.
Thus the premise $(\forall n)(\forall X)(\exists Y)\phi(n,X,Y)$ entails
possession of a uniform procedure for constructing the desired sequence
$(Z_n)$. In our setting the content of the axiom has to do not with choice
but with our ability to perform countable constructions.

The second order comprehension axioms (9)
hold because we take subsets of $\omega$ to be modelled
by (in principle) actually existing infinite sequences of 0's
and 1's: if we knew that $\phi(n) \vee \neg\phi(n)$ held for all $n$
then in principle we could determine the truth value of $\phi(n)$
for every $n$ and use this information to construct a corresponding
$X$. The third order comprehension axioms (10)
are immediately justified by the fact that we take third order
variables to stand for predicates which assign a definite truth
value to the assertion $X \in \Xb$ for any $X$. That is,
we can take $\Xb$ to be the formula $\phi$ itself.

Our conception of second order objects as appearing in well-ordered
stages, such that only countably many of them are available at each
stage, would also support a genuine choice axiom. The most straightforward
way to formalize it would be to augment CM with a second order relation
symbol $\prec$ and add axioms asserting that $\prec$ is a total
ordering, together with the axioms
$$(\forall X)[(\forall Y)(Y \prec X \to \phi(Y)) \to \phi(X)]
\to (\forall X)\phi(X)$$
asserting progressivity of $\prec$. This should only be asserted for
formulas that do not contain $\prec$, for reasons having
to do with the circularity involved in making sense of a
relation that is well-ordered with respect to properties that are
defined in terms of that relation; see Section 2.5 of \cite{W3}. We
could also add the axiom
$$(\forall X)(\exists Z)(\forall Y)[Y \prec X \to (\exists n)(Y = Z_{(n)})]$$
expressing the fact that only countably many second order objects are
available at any moment. We call the resulting system CM${}^+$. None of
the mathematics we develop in Section 3 requires the extra axioms of
CM${}^+$.

We record two general facts about CM that will be of use in the
next section. First, the comprehension axioms immediately imply
{\it arithmetical comprehension}:

\begin{theo}
(a) Let $\phi$ be a formula that contains no free occurences of $X$ and
no second or third order quantifiers. Then CM proves
$$(\exists X)(\forall n)(n \in X \leftrightarrow \phi(n)).$$

(b) Let $\phi$ be a formula that contains no free occurences of $\Xb$ and
no second or third order quantifiers. Then CM proves
$$(\exists \Xb)(\forall X)(X \in \Xb \leftrightarrow \phi(X)).$$
\end{theo}

Second, we have a principle of {\it numerical omniscience}:

\begin{theo}
For any formulas $\phi$ and $\psi$, CM proves
$$(\forall n)(\phi(n) \vee \psi(n)) \to
[(\forall n)\phi(n) \vee (\exists n)\psi(n)].$$
\end{theo}

Intuitively, our ability to perform countable constructions and the
fact that we have a uniform proof of $\phi(n) \vee \psi(n)$ allows
us to verify either $(\forall n)\phi(n)$ or $(\exists n)\psi(n)$.
The theorem is formally proven as follows. Suppose $(\forall n)(\phi(n) \vee
\psi(n))$. Then for any $n$ there exists $X$ such that
$$(\phi(n) \wedge 0 \in X) \vee (\psi(n) \wedge 0 \not\in X).$$
Dependent choice (axiom (8)) followed by arithmetical comprehension then
yields a set $Y$ such that
$$(\phi(n) \wedge n \in Y) \vee (\psi(n) \wedge n \not\in Y)$$
holds for all $n$. Finally, since classical logic holds for formulas
without second or third order quantifiers, we have
$(\forall n)(n \in Y) \vee (\exists n)(n \not\in Y)$,
and we can then infer $(\forall n)\phi(n) \vee (\exists n)\psi(n)$.

As a special case of numerical omniscience we have that
$(\forall n)(\phi(n) \vee \neg\phi(n))$ implies both $(\forall n)\phi(n)
\vee \neg(\forall n)\phi(n)$ and $(\exists n)\phi(n) \vee
\neg(\exists n)\phi(n)$. In effect, this means that any
formula whose truth can be evaluated in countably many steps satisfies
the law of excluded middle. This is important because it implies that
we can use classical logic, in particular proofs by contradiction, in
such cases.

As an alternative formulation of CM we could adopt numerical omniscience
as an axiom and only assume the law of excluded middle for atomic formulas.
It comes to the same thing because numerical omniscience plus excluded
middle for atomic formulas implies excluded middle for all
arithmetical formulas (by induction on formula complexity).

\section{Development of core mathematics}

\subsection{Preliminaries}

The remainder of the paper will sketch how core mathematics can
be developed within the framework presented in Section 2. We concentrate
on analysis because this is the mainstream area that is, broadly speaking,
most resistant to formalization in weak systems. Our development
is similar to the one in \cite{W2}. Probably the main hurdles to overcome
are getting accustomed to the basic set-up established in the
present section and familiarizing oneself with the technique of using
comprehension axioms to prove existence results. In the following we will
reason informally in CM.

\begin{defi}
Let $\tilde{N}$ be a second-order constant denoting the natural numbers
(including zero):
$$(\forall n)(n \in \tilde{N}).$$
\end{defi}

$\tilde{N}$ exists by second order comprehension. CM contains the axioms of
Peano arithmetic, so elementary number theory
in $\tilde{N}$ can be developed as usual.

\begin{defi}
An {\it ordered $k$-tuple of natural numbers} ($k \geq 2$)
is a nonzero natural number
having no prime divisors besides $p_1, \ldots, p_k$, where $p_i$ is
the $i$th prime. We write $\langle a_1, \ldots, a_k\rangle =
p_1^{a_1}\cdots p_k^{a_k}$. For any $X_1, \ldots, X_k$ ($k \geq 2$) we define
$$X_1 \times \cdots \times X_k = \{\langle a_1, \ldots, a_k\rangle:
a_1 \in X_1, \ldots, a_k \in X_k\}.$$
A {\it sequence of second order objects} is a second order object contained
in $\tilde{N}^2 = \tilde{N} \times \tilde{N}$. For such an object $X$ and each
$a \in \tilde{N}$ we write $X_{(a)} = \{b: \langle a,b\rangle \in X\}$.

At the third order level, for any $\Xb_1, \ldots, \Xb_k$ ($k \geq 2$) we define
$\Xb_0 \times \cdots \times \Xb_k$ to be the third order object consisting
of all sequences $X$ of
second order objects such that $X_{(0)} \in \Xb_0, \ldots,
X_{(k)} \in \Xb_k$, and $X_{(l)} = \emptyset$ for all $l > k$;
we write $X = \langle X_{(0)}, \ldots, X_{(k)}\rangle$.
\end{defi}

In the preceding definition $X_1 \times \cdots \times X_k$, $X_{(a)}$,
and $\Xb_0 \times \cdots \times \Xb_k$ all
exist by arithmetical comprehension. Also observe that if
$X = (X_{(n)})$ is a sequence of second order objects then
$\bigcup X_{(n)}$ and $\bigcap X_{(n)}$ exist by arithmetical
comprehension.

We now introduce set language. Notice that in classical set theory quotient
constructions involve passing to a higher type, an avenue that is not
available here. One way to define quotients in the present setting
would be to model the quotient structure by selecting one element from
each equivalence class. This could generally be done using the third order
choice axiom in CM${}^+$ (see Section 2.3), but this method introduces
extraneous information. In other words, quotients become noncanonical.

We opt instead to model a ``set'' as a third order object equipped with
an equivalence relation. This is reflected in the basic definitions that
follow but shows up almost nowhere else: once we agree that we are, in
effect, always working up to equivalence, there is no further need to
explicitly refer to that equivalence.

\begin{defi}
A {\it set} is a third order object $\Xb$ together with a third order
object $\equiv$ contained in $\Xb^2 = \Xb \times \Xb$ which is reflexive,
symmetric, and transitive. We call $\equiv$ the {\it identity} on $\Xb$.
Notationally, we will typically suppress $\equiv$ and refer to $\Xb$ as the
set. A {\it subset} of a set $\Xb$ is a third order object $\Yb$ that is
contained in $\Xb$ and satisfies $X \in \Yb,\, X \equiv Y \Rightarrow Y \in \Yb$,
together with the identity $(\equiv\, \cap \Yb^2)$. We write
$\Yb \subseteq \Xb$. A {\it quotient} of a
set $\Xb$ with identity $\equiv$ is the same third order object $\Xb$
together with an identity $\equiv'$ that contains $\equiv$.
The {\it product} of two sets $\Xb$ and $\Yb$ is the third order
object $\Xb \times \Yb$ together with the identity defined by
$\langle X,Y\rangle \equiv \langle X',Y'\rangle$ if and only if
$X \equiv X'$ and $Y \equiv Y'$.

A {\it relation} on a set $\Xb$ is a subset of $\Xb^2$. A {\it function}
from a set $\Xb$ to a set $\Yb$ is a subset $\fb$ of $\Xb \times
\Yb$ such that for every $X \in \Xb$ there exists exactly one (up to
identity) $Y \in \Yb$ with $\langle X,Y\rangle \in \fb$.
\end{defi}

The next proposition should help orient the reader to the kind of thinking
involved in reasoning in CM.

\begin{prop}
Let $\Xb$ and $\Yb$ be sets, let $\fb: \Xb \to \Yb$ be a function, and let
$\Yb_0$ be a subset of $\Yb$. Then $\fb^{-1}(\Yb_0)$ exists and is a subset
of $\Xb$.
\end{prop}

\begin{proof}
We show that the condition ``$\fb(X) \in \Yb_0$'' satisfies excluded middle.
By third order comprehension this implies that $\fb^{-1}(\Yb_0) =
\{X \in \Xb: \fb(X) \in \Yb_0\}$ exists; the fact that it is compatible
with the identity on $\Xb$
is an easy consequence of the fact that $\fb$ is a subset of $\Xb \times \Yb$
together with the definition of the identity on $\Xb\times \Yb$.

To verify the claim, let $X \in \Xb$. Then by the definition of a function
there exists $Y \in \Yb$, unique up to identity, such that $\langle X,Y\rangle
\in \fb$. Since we assume excluded middle for arithmetical formulas, we have
$(Y \in \Yb_0) \vee (Y \not\in \Yb_0)$, and since $\Yb_0$ is a subset of $\Yb$
this truth value does not depend on the choice of $Y$ up to identity. Therefore
$\fb(X) \in \Yb_0$ satisfies excluded middle.
\end{proof}

In contrast, the push-forward $\fb(\Xb_0)$ need not exist in general because
we have no way to effectively test whether a given $Y \in \Yb$ belongs to the
image. The image of $\Xb_0$ is continually expanding as new second order
objects appear, and given $Y \in \Yb$ we may not be able to predict whether
some future $X \in \Xb$ will map to $Y$. (Inverse images also expand, but
given $X \in \Xb$ the fact that $\fb$ is a function means that we can
construct $\fb(X)$ and we can then immediately check wither $\fb(X)$ belongs
to $\Yb_0$. Further expansion of $\Yb_0$ cannot affect this result.)

However, if $\Xb_0$ is countable (see Definition 3.8 below) then we can use
numerical omniscience to check whether $Y$ belongs to $\fb(\Xb_0)$, so images
of countable sets do always exist (Proposition 3.9).

\subsection{$\tilde{Z}$ and $\tilde{Q}$}

We encode the integers using a sign bit (0 for $+$ and 1 for $-$):

\begin{defi}
Let $\tilde{Z}$ be a second order constant denoting the set of all ordered
pairs of natural numbers $\langle a,b\rangle$ such that either $a = b = 0$ or
$a > 0$ and $b = 0$ or $1$. We define addition in $\tilde{Z}$ by letting
$\langle a,b\rangle + \langle a',b'\rangle$ be

\begin{tabular}{l l}
$\langle a + a', b\rangle$ &if $b = b'$\cr
$\langle a - a', 0\rangle$ &if $b = 0$, $b' = 1$, and $a \geq a'$\cr
$\langle a' - a, 1\rangle$ &if $b = 0$, $b' = 1$, and $a < a'$\cr
$\langle a - a', 1\rangle$ &if $b = 1$, $b' = 0$, and $a > a'$\cr
$\langle a' - a, 0\rangle$ &if $b = 1$, $b' = 0$, and $a \leq a'$.\cr
\end{tabular}

\noindent The product and the order relation are defined by cases in the
same way.
\end{defi}

The existence of $\tilde{Z}$ is provable in CM by arithmetical comprehension.
Basic properties of $\tilde{Z}$ as an ordered ring are straightforwardly
provable as facts of first order arithmetic.

We define the rationals as fractions in lowest terms.

\begin{defi}
Let $\tilde{Q}$ be a second order constant denoting
the set of all ordered pairs $\langle a,b\rangle$ with $a,b \in \tilde{Z}$
relatively prime and $b$ positive.
\end{defi}

Here ``relatively prime'' and ``positive'' mean with respect to the product
and order on $\tilde{Z}$. $\tilde{Q}$ exists by arithmetical comprehension.

We define order, addition, multiplication, and division (with nonzero
denominator) in $\tilde{Q}$ in the usual way.
All these definitions are arithmetical and basic
properties of $\tilde{Q}$ as an ordered field are straightforwardly provable
as facts of first order arithmetic. (For more details on
the material of this section, with inessentially different definitions, see
Section II.4 of \cite{Sim}.)

\subsection{The real line}

We now define the real line in terms of Dedekind cuts of $\tilde{Q}$.

\begin{defi}
Let $\Rb$ be a third order constant satisfying $X \in \Rb$ if and
only if $X \subseteq \tilde{Q}$, $\emptyset \neq X \neq \tilde{Q}$,
$X$ has no greatest element, and
$$(p \in X,\, q \in \tilde{Q},\, q < p) \Rightarrow q \in X$$
(where $<$ is the order relation defined on $\tilde{Q}$).
We equip $\Rb$ with the trivial identity $x \equiv y \Leftrightarrow
x = y$.
\end{defi}

$\Rb$ contains canonical copies of $\tilde{N}$, $\tilde{Z}$, and
$\tilde{Q}$, which we denote $N$, $Z$, and $Q$. For instance, $X \in Q$
holds if and only if there exists $p \in \tilde{Q}$ such that $q \in X
\Leftrightarrow q < p$.

Note that by arithmetical comprehension, for any $\tilde{X} \subseteq
\tilde{Q}$ there exists $X \subseteq Q$ containing the corresponding elements,
and vice versa. Similar statements hold for $N$ and $Z$. Thus, we can
effectively identify $N$ with $\tilde{N}$, $Z$ with $\tilde{Z}$, and
$Q$ with $\tilde{Q}$, and we will generally do so without comment.

Now that we have a third order version of $N$, we can make the following
definitions:

\begin{defi}
A set $\Xb$ is {\it countable} if there exists a surjective
function $\fb$ from $N$ to $\Xb$.
A {\it sequence of subsets of $\Xb$} is a subset $\Yb$ of
$N \times \Xb$; we write
$$\Yb_{(n)} = \{X: \langle n, X\rangle \in \Yb\}$$
(and set $X \equiv Y$ in $\Yb_{(n)}$ if $\langle n,X \rangle \equiv
\langle n,Y\rangle$ in $\Yb$).
The {\it product} $\prod \Yb_{(n)}$ of a sequence $\Yb$ of subsets of a
set is the set of all sequences $Y$ of second order objects such that
$Y_{(n)} \in \Yb_{(n)}$ for all $n$, with $Y \equiv Y'$ if
$Y_{(n)} \equiv Y'_{(n)}$ for all $n$.
\end{defi}

For any $n \in \tilde{N}$ and $X \subseteq \tilde{N}$ the ordered pair
$\langle n,X\rangle$ exists by arithmetical comprehension.  It follows
that if $\Yb$ is a sequence of subsets of $\Xb$ then $X \in \Yb_{(n)}
\vee X \not\in \Yb_{(n)}$, for any $n$ and $X$.
From this one easily sees (using third order comprehension) that
$\bigcup \Yb_{(n)}$ and $\bigcap \Yb_{(n)}$ exist.

\begin{prop}
Let $\fb: \Xb \to \Yb$ be a function and let $\Xb_0 \subseteq \Xb$ be
countable. Then $\fb(\Xb_0)$ exists and is a subset of $\Yb$.
\end{prop}

\begin{proof}
We can verify excluded middle for the condition $(\exists X)(X \in \Xb
\wedge Y = \fb(X))$ by numerical omniscience.
\end{proof}

The order relation and the algebraic operations on $\Rb$ are easily
(arithmetically) defined. Also, if $(X_n)$ is a sequence of reals that
is bounded above then $\bigcup X_n \in \Rb$ is its least upper bound.
So $\Rb$ is, in this sense,
a sequentially complete ordered field. The converse is also true:

\begin{theo}
$\Rb$ is a sequentially complete ordered field. Every sequentially
complete ordered field is isomorphic to $\Rb$.
\end{theo}

The second statement of the theorem is proven in the usual way:
given any sequentially complete ordered field $\Fb$, first isolate
the countable subfield $F$ generated by $1$; then establish an
isomorphism between $F$ and $Q$; and finally show that $F$
is dense in $\Fb$ and
use this to define the desired isomorphism between $\Fb$ and $\Rb$.
The novelty here is that the existence of $F$ requires proof. The
result we need is stated in the following lemma.

\begin{lemma}
Any countable subset of a field $\Fb$ generates a countable subfield of $\Fb$.
\end{lemma}

\begin{proof}
Let $S \subseteq \Fb$ be countable. The lemma does not follow from
arithmetical comprehension; we must prove directly that the condition
``$X$ is in the subfield generated by $S$'' is expressible in a way
that satisfies the law of excluded middle, so that third order
comprehension can be used. Informally, we expand the above
to ``there is a word in elements of $S$ which when evaluated
in $\Fb$ produces $X$, up to identity.'' Next we use induction
on word length to check that for any word $w$ the statement
``evaluating $w$ in $\Fb$ produces $X$, up to identity'' satisfies
excluded middle. Enumerating the
words as $(w_n)$, second order comprehension then verifies the existence
of a sequence $(X_n)$ such that $X_n$ is the result of evaluating $w_n$
in $\Fb$ (or $X_n = 0$ if $w_n$ does not evaluate). Finally, numerical
omniscience implies the law of excluded
middle for the assertion $(\exists n)(X \equiv X_n)$. Together with third
order comprehension, this shows that the subfield of $\Fb$ generated
by $S$ exists.
\end{proof}

The same argument will apply in more general algebraic settings to show
the existence of countably generated subobjects.

The preceding proof is a good illustration of the technique of combining
numerical omniscience with comprehension for existence results. This will
be used repeatedly in the sequel. The general principle is: if we can
test whether $\phi(X)$ holds in countably many steps then $\{X: \phi(X)\}$
exists.

\begin{theo}
$\Rb$ is Cauchy complete.
\end{theo}

\begin{proof}
This follows from the formula
$$\liminf X_n = \bigcup_{n \in N} \bigcap_{k \geq n} (X_k - 1/n),$$
where $X_k - 1/n \in \Rb$ is taken literally as a Dedekind cut.
\end{proof}

\begin{prop}
Let $\Xb \subseteq \Rb$. Then any sequence of functions
$\fb_n: \Xb \to \Rb$ which converges pointwise has a pointwise
limit $\fb: \Xb \to \Rb$.
\end{prop}

\begin{proof}
By arithmetical comprehension.
\end{proof}

Polynomial functions from $\Rb$ to itself are arithmetically
definable. It easily follows that all of the standard continuous functions
from real analysis ($\sin x$, $\cos x$, $e^x$, $\ln x$, etc.) may
be defined. Standard discontinuous functions such as the Heaviside step
function, the comb function, and the characteristic function of the
Cantor set are also straightforwardly definable.

\subsection{Topology in $\Rb$}

Let $Q^+ = \{p \in Q: p > 0\}$ and $\Rb^+ = \{x \in \Rb: x > 0\}$.
(We will now start using lowercase letters to denote elements of $\Rb$.)

\begin{defi}
A subset $\Ub \subseteq \Rb$ is {\it open} if there is a set
$P \subseteq Q \times Q^+$ such that $x \in \Ub$ if and only if
$|p - x| < r$ for some $\langle p,r \rangle \in P$. We call $P$ a
{\it witness} for $\Ub$.

An {\it open ball} in $\Rb$ is a set of the form
$${\bf ball}_r(x) = \{y \in \Rb: |x - y| < r\}$$
for some $\langle x,r\rangle \in \Rb \times \Rb^+$.
\end{defi}

\begin{prop}
(a) Open balls are open.

(b) The union of any sequence of open subsets of $\Rb$ is open.

(c) The intersection of any finitely many open subsets of $\Rb$ is open.
\end{prop}

\begin{proof}
Parts (a) and (c) are routine. In part (b), given a sequence $(\Ub_n)$
of open sets we need to use dependent choice to select a sequence of
witnesses $(P_n)$; the union of this sequence is then a witness for
$\bigcup \Ub_n$.
\end{proof}

\begin{defi}
The {\it closure} of a countable set $C \subset \Rb$ is the set
$$\{x \in \Rb:\hbox{ for every }n\hbox{ there
exists }y \in C\hbox{ such that }|x - y| < 1/n\}.$$
A subset $\Cb$ of $\Rb$ is {\it closed} if it is the closure of a
countable set $C \subset \Rb$; we call $C$ a {\it witness} for $\Cb$.
\end{defi}

In the definition of closure, observe that for any $x$ the existence,
for every $n \in N$, of a point $y \in C$ such that $|x - y| < 1/n$
satisfies the law of excluded middle by numerical omniscience, because
$N$ and $C$ are countable. Thus the closure of $C$ exists by third
order comprehension. This technique was already introduced in the
proof of Lemma 3.11 and from now on we will use it without comment.

\begin{prop}
Closed subsets of $\Rb$ are sequentially closed (i.e., contain all
limits of Cauchy sequences).
\end{prop}

\begin{proof}
Let $\Cb$ be the closure of a countable set $C$ and enumerate the
elements of $C$ as $(x_k)$. Now let $(y_n)$ be a Cauchy sequence in
$\Cb$. For each $n$ let $k_n$ be the smallest index such that
$|y_n - x_{k_n}| < 1/n$; then the sequence $(x_{k_n})$ is Cauchy
and converges to the same limit as $(y_n)$, hence that limit belongs
to $\Cb$.
\end{proof}

\begin{theo}
A subset of $\Rb$ is closed if and only if its complement is open.
\end{theo}

\begin{proof}
The forward direction is easy: given a witness $C$ for a closed set, the
set of pairs $\langle p,r\rangle \in Q \times Q^+$
such that $|p - x| \geq r$ for all $x \in C$ is
a witness for the complementary open set. For the reverse direction,
given a witness $P$ for an open set $\Ub$ we construct a witness $C$
for the complementary closed set $\Cb$ as follows. First let $C_0$ be
the set of rationals in $\Rb - \Ub$, i.e., the set of $p' \in Q$ such
that $|p' - p| \geq r$ for all $\langle p,r\rangle \in P$.
Then for each rational $p \in \Ub$
let $S_p$ be the set of $p' \in Q$ such that there is a finite
sequence of elements $\langle p_i, r_i \rangle \in P$, $1 \leq i \leq n$,
with $p \in {\bf ball}_{r_1}(p_1)$, $p' \in {\bf ball}_{r_n}(p_n)$, and
${\bf ball}_{r_i}(p_i) \cap {\bf ball}_{r_{i+1}}(p_{i+1}) \neq \emptyset$ for
$1 \leq i < n$. If $S_p$ is bounded above then let $x_p$ be its least
upper bound; thus $x_p$ is the smallest element of $\Cb$ that is
greater than $p$. Finally, let $C_1$ be the set of all such elements $x_p$
and let $C = C_0 \cup C_1$. One easily checks that $\Cb$ is the closure of
$C$.
\end{proof}

\begin{coro}
(a) The intersection of any sequence of closed subsets of $\Rb$ is closed.

(b) The union of any finitely many closed subsets of $\Rb$ is closed.
\end{coro}

\begin{prop}
Let $\Xb \subseteq \Rb$ and let $X \subseteq \Xb$ be countable.
Then every open ball about any point in $\Xb$ intersects $X$ if and
only if $\Xb$ is contained in the closure of $X$.
\end{prop}

\begin{defi}
We say that $\Xb \subseteq \Rb$ is {\it separable} if it has a
countable subset $X$ such that either of the two equivalent conditions
in Proposition 3.20 is satisfied. We say that $X$ is {\it dense} in $\Xb$.
\end{defi}

\begin{lemma}
Every open subset of $\Rb$ is separable, as is every closed subset of $\Rb$.
\end{lemma}

\begin{prop}
Let $\Xb \subseteq \Rb$ be a separable subset and let $\Yb =
\Rb - \Xb$.

(a) $\Xb$ is closed if and only if it is closed under limits of
Cauchy sequences.

(b) $\Yb$ is open if and only if for every $x \in \Yb$ there
exists $r > 0$ such that ${\bf ball}_r(x) \subseteq \Yb$.
\end{prop}

\begin{proof}
(a) The forward direction was Proposition 3.17. For the reverse
direction suppose $\Xb$ is sequentially closed, let $C$ be a countable
dense subset of $\Xb$, and let
$\Cb$ be the closure of $C$. Then $\Xb$ is contained in $\Cb$ by
density and $\Xb$ contains $\Cb$ since it is closed under Cauchy
convergence. So $\Xb = \Cb$ and hence $\Xb$ is closed.

(b) The forward direction is trivial; for the reverse direction
let $C$ be a countable dense subset of $\Xb = \Rb - \Yb$ and
verify that $\Yb$ is disjoint from the closure of $C$. It follows
that $\Xb$ is the closure of $C$, hence $\Xb$ is closed, hence
$\Yb$ is open.
\end{proof}

\begin{defi}
$\Kb \subseteq \Rb$ is {\it compact} if every
sequence of open sets that covers $\Kb$ has a finite subcover.
\end{defi}

\begin{theo}
Let $\Kb$ be a separable subset of $\Rb$. Then the following
are equivalent:

(i) $\Kb$ is closed and bounded;

(ii) $\Kb$ is compact;

(iii) $\Kb$ is bounded and contains the limits of all of its
Cauchy sequences;

(iv) every sequence in $\Kb$ has a subsequence which converges
to a limit in $\Kb$.
\end{theo}

\begin{proof}
The proofs of (i) $\Rightarrow$ (ii) $\Rightarrow$ (iii)
$\Rightarrow$ (iv) are standard and do not use separability.
For (iv) $\Rightarrow$ (i), suppose every sequence has a convergent
subsequence and let $C$ be a countable dense subset of $\Kb$. Then by
numerical omniscience the assertion ``$C$ is bounded'' satisfies excluded
middle, so a proof by contradiction shows that $C$, and hence $\Kb$, must
be bounded. The fact that $\Kb$ is closed follows from Proposition 3.23 (a).
\end{proof}

\begin{defi}
Let $\Xb \subseteq \Rb$. We say that a function $\fb: \Xb \to \Rb$
is {\it continuous} if the inverse image of any open set in $\Rb$ is
the intersection of an open subset of $\Rb$ with $\Xb$.
\end{defi}

\begin{theo}
Suppose $\Xb \subseteq \Rb$ is separable and let $\fb: \Xb \to \Rb$
be a function. Then the following are equivalent:

(i) $\fb$ is continuous;

(ii) the inverse image of every closed set in $\Rb$ is the
intersection of a closed set in $\Rb$ with $\Xb$;

(iii) for any countable set $C \subseteq \Xb$ with closure
$\overline{C}$ we have $x \in \overline{C} \cap \Xb$ $\Rightarrow$
$f(x) \in \overline{\fb(C)}$;

(iv) $\fb$ preserves convergence of sequences;

(v) for every $x \in \Xb$ and every $\epsilon > 0$ there exists
$\delta > 0$ such that $\db(x, y) < \delta$ implies $\db(\fb(x),\fb(y))
< \epsilon$.
\end{theo}

\begin{proof}
The proofs of (i) $\Rightarrow$ (ii) $\Rightarrow$ (iii) $\Rightarrow$
(iv) are standard and do not use separability. For (iv) $\Rightarrow$
(v), let $X$ be a countable dense subset of $\Xb$. We first show
that the $\epsilon$-$\delta$ condition holds for $x$ and $y$ in $X$ and with
$\epsilon$ and $\delta$ restricted to rational values. If not then
we can find $x \in X$ and a sequence $(y_n) \subseteq X$ such
that $y_n \to x$ but $\fb(y_n) \not\to \fb(x)$, contradicting (iv).
Since $X$ is countable and $\epsilon$ and $\delta$ are restricted
to the rationals we have excluded middle, so we conclude that the
$\epsilon$-$\delta$ condition does hold for $x$ in $X$. Since $X$ is
dense in $\Xb$ and $\fb$ preserves convergence of sequences, (v)
follows easily.

For (v) $\Rightarrow$ (i), again let $X$ be a countable dense subset
of $\Xb$. Also let $\Ub \subseteq \Rb$ be an open set with witness $P$.
Let $P'$ be the set of pairs $\langle p', r'\rangle \in Q \times Q^+$
such that $\fb(X \cap {\bf ball}_{r'}(p')) \subseteq {\bf ball}_{r - \epsilon}(p)$
for some $\langle p,r\rangle \in P$ and some $\epsilon > 0$. It is
then straightforward to verify that $P'$ is a witness for an open
set $\Ub'$ which satisfies $\Ub' \cap \Xb = \fb^{-1}(\Ub)$.
\end{proof}

\begin{theo}
The sum and product of two continuous functions from $\Xb \subseteq \Rb$
to $\Rb$ are continuous. The composition of two continuous functions, if
defined, is continuous.
\end{theo}

\begin{theo}
Let $\Xb \subseteq \Rb$ and let $\fb: \Xb \to \Rb$
be continuous. If $\Xb$ is separable and compact then $\fb$ is
bounded and achieves its maximum and minimum. If $\Xb$ contains
the interval $[a,b]$ then $\fb$ attains every value between $\fb(a)$
and $\fb(b)$.
\end{theo}

\begin{proof}
For the first statement, let $X = (x_n)$ be a countable dense subset
of $\Xb$, pass to a subsequence $(x_{n_k})$ such that $\fb(x_{n_k})$
converges to $\sup \fb(x_n)$ or $\inf \fb(x_n)$ (possibly $\pm \infty$),
then pass to a subsequence which converges in $X$, and finally
apply Theorem 3.27 (iv) (showing $\pm \infty$ are not possible
maximum and minimum values).

For the second statement, suppose $\fb(a) \neq \fb(b)$ and let $z$ be
any value strictly between $\fb(a)$ and $\fb(b)$. Consider the disjoint
open sets $\Ub = \fb^{-1}((-\infty, z))$ and $\Vb = \fb^{-1}((z,\infty))$.
Without loss of generality suppose $a \in \Ub$ and $b \in \Vb$.
Let $Y$ be the set of rationals $p > a$ such that every rational in
$[a,p]$ lies in $\Ub$; then $x = \sup Y$ cannot lie in $\Ub$ (since $\Ub$ is
open) or in $\Vb$ (since $\Vb$ is open and disjoint from $\Ub$). Also
$a < x < b$, so that $x \in \Xb$. Since
$\fb(x) \not\in \Rb - \{z\}$ we must have $\fb(x) = z$.
\end{proof}

\subsection{Metric spaces}

\begin{defi}
A {\it metric space} is a set $\Xb$ together with a function
$\db: \Xb\times\Xb \to [0,\infty)$ which satisfies the usual
metric axioms. It is {\it complete} if every Cauchy sequence converges.
A subset is {\it dense} if it intersects every {\it open ball}
${\bf ball}_r(x) = \{y \in \Xb: \db(x,y) < r\}$ and a space is
{\it separable} if it contains a countable dense subset.
\end{defi}

\begin{prop}
Every metric space densely embeds in a complete metric space. This
embedding is unique up to an isometric isomorphism fixing the
original space.
\end{prop}

\begin{proof}
The proof is essentially the standard one. The set of Cauchy sequences
in a metric space exists by third order comprehension, as does
the standard equivalence relation on Cauchy sequences. We can then
use this equivalence relation as the identity in the set of Cauchy
sequences. The remainder of the proof is standard.
\end{proof}

This is our first serious use of the convention that identity in
sets is determined by equivalence relations.

A subtle point: although there is a function $\fb$ that embeds $\Xb$ into
its completion, the image $\fb(\Xb)$ need not exist as a set. We noted
earlier (at the end of Section 3.1) that images of sets do not exist
in general, and this is true even in the present rather special situation.

\begin{defi}
Let $\Xb$ be a metric space. A subset $\Ub \subseteq \Xb$ is {\it open}
if there is a countable set $P \subseteq \Xb \times \Rb^+$ such that
$\Ub = \bigcup_{\langle x,r\rangle \in P} {\bf ball}_r(x)$. A subset $\Cb \subseteq \Xb$
is {\it closed} if its complement is open. We call $P$ a {\it witness} both
for the open set $\Ub$ and the complementary closed set $\Xb - \Ub$.
\end{defi}

\begin{theo}
(a) Open balls are open.

(b) The union of any sequence of open subsets of a metric
space is open.

(c) The intersection of any finitely many open subsets of a separable metric
space is open.
\end{theo}

\begin{proof}
The only subtlety occurs in part (c), which comes down to showing that
the intersection of any two open balls is open. We obtain a witness for
${\bf ball}_r(x) \cap {\bf ball}_{r'}(x')$ by letting $X$ be a countable dense subset
and taking the set of pairs $\langle y,s\rangle$ such that $y \in X$,
$s \in Q^+$, $\db(x,y) + s \leq r$, and $\db(x',y) + s \leq r'$.
\end{proof}

\begin{coro}
(a) The intersection of any sequence of closed subsets of a metric
space is closed.

(b) The union of any finitely many closed subsets of a separable metric
space is closed.
\end{coro}

\begin{defi}
The {\it closure} of a countable set $C \subseteq \Xb$ is the set of
all limits of convergent sequences in $C$. $\Xb$ is {\it totally bounded}
if for every $n \in N$ there is a finite set $S \subseteq \Xb$ such that
every $x \in \Xb$ satisfies $\db(x,s) < 1/n$ for some $s \in S$. $\Xb$ is
{\it compact} if every sequence of closed sets, any finitely many of which
have nonempty intersection, has nonempty intersection. $\Xb$ is {\it
boundedly compact} if every {\it closed ball} ${\overline{{\bf ball}}}_r(x) =
\{y \in \Xb: \db(x,y) \leq r\}$ is compact.
\end{defi}

(Closures exist by third order comprehension and a double application
of numerical omniscience: first we use it to check that for any $x \in \Xb$
and any $n \in N$ the condition ${\bf ball}_{1/n}(x) \cap C \neq \emptyset$
satisfies excluded middle, then we use it again to check that
the condition $(\forall n)({\bf ball}_{1/n}(x) \cap C \neq \emptyset)$ satisfies
excluded middle.)

\begin{theo}
Let $\Xb$ be a separable metric space. Then the following are equivalent:

(i) $\Xb$ is compact;

(ii) $\Xb$ is complete and totally bounded;

(iii) every sequence in $\Xb$ has a convergent subsequence.
\end{theo}

\begin{proof}
(i) $\Rightarrow$ (ii): Let $X$ be a countable dense subset of $\Xb$.
Suppose $\Xb$ is compact and let $(x_n)$ be a Cauchy sequence in $\Xb$.
For each $k$ let $n_k$ be the smallest natural number such that
$\db(x_n, x_{n_k}) \leq 1/k$ for all $n > n_k$. Then the set of pairs
$\langle x,q\rangle$ with $x \in X$, $q \in Q^+$, and $\db(x,x_{n_k})
> q + 1/k$ witnesses an open set $\Ub_k$. Applying the compactness
hypothesis to the sequence of complementary closed sets then produces
a limit for $(x_n)$. This shows that $\Xb$ is complete. To verify total
boundedness, enumerate the elements of $X$ (up to identity) as $(z_n)$
and observe that the assertion ``for
every $k$ there exists $n$ such that every $z_i$ is within $1/k$ of
some $z_j$ with $1 \leq j \leq n$'' satisfies excluded middle. So
suppose this statement fails. Then there exists $k$ such that for
every $n$, some $z_i$ satisfies $\db(z_i,z_j) \geq 1/k$ for
$1 \leq j \leq n$. Using dependent choice, we can then construct a
sequence $(n_k)$ such that $\db(z_{n_i}, z_{n_j}) \geq 1/k$ for all
$i \neq j$. Finally, for each $i$ the pairs $\langle x,q \rangle$
with $x \in X$, $q \in Q^+$, and $q < \db(x,z_{n_j}) - 1/2k$ for all
$j \geq i$ witness an open set $\Ub_i$, and the complementary closed
sets falsify the compactness condition. Thus $\Xb$ must be totally
bounded.

(ii) $\Rightarrow$ (iii): Assume (ii) and let $(x_n)$ be any sequence
in $\Xb$. By completeness it will suffice to show that $(x_n)$ has a
Cauchy subsequence. Let $S_1 \subseteq \Xb$ be a finite set such that
for all $x \in \Xb$ there exists $s \in S_1$ with $\db(x,s) < 1/2$.
Find $s_1 \in S_1$ such that
$\db(x_n,s_1) < 1/2$ for infinitely many $n$, and let $n_1$ be the
smallest number such that $\db(x_{n_1},s_1) < 1/2$. Then let $S_2 \subseteq
\Xb$ be a finite set such that for all $x \in \Xb$ there exists $s \in S_2$
with $\db(x,s) < 1/4$,
find $s_2 \in S_2$ such that $\db(s_1,s_2) < 1/2 + 1/4$ and $\db(x_n,s_2)
< 1/4$ for infinitely many $n$, and let $n_2$ be the smallest number
after $n_1$ such that $\db(x_{n_2},s_2) < 1/4$. Continue in this way,
with $\db(s_k, s_{k+1}) < 2^{-k} + 2^{-k-1}$ and $\db(x_{n_k}, s_k)
< 2^{-k}$. Then $(x_{n_k})$ is the desired Cauchy subsequence.

(iii) $\Rightarrow$ (i): Assume (iii). Let $(\Cb_n)$ be a sequence of closed
subsets with the finite intersection property, and for each $n$
choose a point $x_n \in \bigcap_{k=1}^n \Cb_k$. Then $(x_n)$ has a
convergent subsequence, and the limit of this sequence belongs to
every $\Cb_n$.
\end{proof}

\begin{prop}
If $\Xb$ is separable then the closure of any countable set is
closed. If $\Xb$ is separable and boundedly compact then every
closed set is separable.
\end{prop}

\begin{proof}
The first statement is easy: let $X$ be a countable dense subset of
$\Xb$; then a witness for the closure of any countable set $C$ is
given by the pairs $\langle x,r\rangle \in X\times Q^+$ such that
$\db(x,y) \geq r$ for all $y \in C$. For the second statement suppose
$\Xb$ is also boundedly compact and let $\Cb \subseteq \Xb$ be closed.
We construct, for each $x \in X$ and $r \in Q^+$ such that
$\overline{{\bf ball}}_r(x)$ intersects $\Cb$, an element of
$\overline{{\bf ball}}_{r}(x) \cap \Cb$.
This produces a countable subset of $\Cb$ that is evidently dense in
$\Cb$. To do this, fix a witness $P$ for $\Cb$ and enumerate $P$ as
$(\langle x_n,r_n\rangle)$. Let $R$ be the set of pairs
$\langle x, r\rangle \in X \times Q^+$ such that for any $n$ there exists
$y \in X \cap {\bf ball}_{r + 1/n}(x)$ with $\db(x_i,y) \geq r_i - 1/n$
for $1 \leq i \leq n$. This set exists by numerical omniscience. Observe
that if $\langle x,r\rangle \not\in R$ then $\overline{{\bf ball}}_r(x) \cap \Cb
= \emptyset$. But if $\langle x,r\rangle \in R$ then we can find a sequence
$(y_n)$ such that $\db(x,y_n) < r + 1/n$ and $\db(x_i,y_n) \geq r_i - 1/n$
for $1 \leq i \leq n$. Letting $y$ be a limit point of this sequence
(using bounded compactness and Theorem 3.36 (iii)), we must have
$\db(x_i,y) \geq r_i$ for all $i$, i.e., $y \in \Cb$. By dependent choice
we can select one such $y$ for each pair $\langle x,r\rangle \in R$; this
is the desired countable dense subset of $\Cb$.
\end{proof}

\begin{defi}
A function $\fb: \Xb \to \Yb$ between metric spaces is {\it continuous}
if the inverse image of any open set in $\Yb$ is open in $\Xb$. It is
a {\it homeomorphism} if it is a bijection and its inverse is also
continuous.
\end{defi}

\begin{theo}
Let $\Xb$ and $\Yb$ be metric spaces, suppose $\Xb$ is separable,
and let $\fb: \Xb \to \Yb$ be a function. Then the following are
equivalent:

(i) $\fb$ is continuous;

(ii) the inverse image of every closed set in $\Yb$ is closed in $\Xb$;

(iii) for any countable set $C \subseteq \Xb$ with closure
$\overline{C}$ we have $x \in \overline{C}$ $\Rightarrow$
$f(x) \in \overline{\fb(C)}$;

(iv) $\fb$ preserves convergence of sequences;

(v) for every $x \in \Xb$ and every $\epsilon > 0$ there exists
$\delta > 0$ such that $\db(x, y) < \delta$ implies $\db(\fb(x),\fb(y))
< \epsilon$.
\end{theo}

(The proof is a straightfoward generalization of the proof of Theorem 3.27.)

\begin{prop}
Let $\Xb$ and $\Yb$ be metric spaces and suppose $\Xb$ is compact.

(a) Every closed subset of $\Xb$ is compact.

(b) If $\Xb$ is separable and $\fb: \Xb \to \Yb$ is continuous
then $\fb(\Xb)$ exists and is a separable compact subset of $\Yb$.

(c) If $\Xb$ is separable and $\fb: \Xb \to \Yb$ is a continuous
bijection then it is a homeomorphism.
\end{prop}

\begin{proof}
Part (a) is trivial since every closed subset of a closed subset of
$\Xb$ is closed in $\Xb$. For part (b) let $X$ be a countable dense subset
of $\Xb$ and let $\Cb$ be the closure of $\fb(X)$. Theorem 3.39 (iii)
implies that $\fb$ maps every element of $\Xb$ into $\Cb$, and Theorem
3.36 (iii) plus Theorem 3.39 (iv) implies that every element of $\Cb$
is in the image of $\fb$. So $\fb(\Xb)$ exists and equals $\Cb$. $\Cb$ is
clearly separable, and compactness follows easily from Theorem 3.39 (ii)
(considering $\fb$ as a function from $\Xb$ to $\Cb$). Part (c) is proven
by combining parts (a) and (b) with both parts of Proposition 3.37,
using the charaterization of continuity of $\fb^{-1}$ in Theorem 3.39 (ii).
($Y$ is separable because it is the closure of $\fb(X)$, as in part (b).)
\end{proof}

\begin{theo}
The intersection of any sequence of open dense subsets of a
separable complete metric space is dense.
\end{theo}

(The proof is identical to the classical proof.)

\subsection{Topological spaces}

We introduce the notion of a family of subsets of a set.

\begin{defi}
A {\it family of subsets} of a set $\Xb$ is a subset $\Tc$ of $\Xb\times\Tb$
for some set $\Tb$. For each $Y \in \Tb$ we write $\Tc_{(Y)} =
\{x \in \Xb: \langle x,Y\rangle \in \Tc\}$. We say that $\Yb$ {\it belongs
to} the family $\Tc$ if $\Yb = \Tc_{(Y)}$ for some $Y \in \Tb$.

A {\it topological space} is a set $\Xb$ together with a family of
subsets $\Tc$ of $\Xb$ such that (i) $\emptyset$ and $\Xb$ belong to
$\Tc$; (ii) the union of any sequence of sets that belong to $\Tc$
belongs to $\Tc$; and (iii) the intersection of any finitely many sets
that belong to $\Tc$ belongs to $\Tc$. $\Tc$ is a {\it topology} on $\Xb$.

A subset of a topological space is {\it open} if it belongs to $\Tc$
and {\it closed} if its complement belongs to $\Tc$.
\end{defi}

\begin{defi}
Let $\Xb$ be a topological space with topology $\Tc \subseteq \Xb \times \Tb$
and let $\Yb \subseteq \Xb$. The {\it relative topology} $\Tc'$ on $\Yb$
is the family $\Tc' = \Tc \cap (\Yb \times \Tb)$.
\end{defi}

It is easy to see that $\Tc'$ is a topology on $\Yb$.

Next we indicate how topologies can be generated from bases.

\begin{prop}
Let $\Xb$ be a set and let $\Bc \subseteq \Xb \times \Bb$ be a family
of subsets of $\Xb$
such that $\emptyset$ and $\Xb$ belong to $\Bc$ and the intersection
of any two sets that belong to $\Bc$ is the union of a sequence of
sets that belong to $\Bc$. Let $\Tb$ be the set of $Y$ such that
$Y_{(n)} \in \Bb$ for all $n$ and let $\Tc$ be the set of pairs
$\langle x,Y\rangle$ such that $x \in \bigcup_n \Bc_{(Y_{(n)})}$.
Then $\Tc$ is a topology on $\Xb$.
\end{prop}

\begin{defi}
The family $\Bc$ in Proposition 3.44 is a {\it base} for the topology $\Tc$.
\end{defi}

\begin{prop}
Let $\Xb$ be a separable metric space with countable dense subset $X$.
Let $\Bb = X \times Q^+$ and let $\Bc \subseteq \Xb \times \Bb$ be the
set of pairs $\langle y,\langle x,r\rangle\rangle$ such that
$\db(x,y) < r$. Then $\Bc$ is a base for a topology and the
open sets are precisely those identified in Definition 3.32.
\end{prop}

\begin{defi}
Let $(\Xb_n)$ be a sequence of sets and let $(\Tc^n)$ be
a corresponding sequence of topologies. Let $\Bb$ consist of all pairs
$\langle m,Y\rangle$ such that $m \in N$ and $Y \in \Tb^1 \times \cdots
\times \Tb^m$. The {\it product topology} on the product $\prod \Xb_n$ is
the topology generated by the base $\Bc \subseteq (\prod \Xb_n) \times \Bb$
consisting of all pairs $\langle x,\langle m,Y\rangle\rangle$ such that
$x_{(n)} \in \Tc^n_{(Y_{(n)})}$ for all $n \leq m$.
\end{defi}

\begin{defi}
A function between topological spaces is {\it continuous} if the
inverse image of any open set is open.
\end{defi}

\begin{prop}
The composition of two continuous functions is continuous.
\end{prop}

\begin{prop}
Let $\Xb$ be a topological space and let $(\Xb_n)$ be a sequence of
topological spaces. Then a function $\fb: \Xb \to \prod \Xb_n$ is
continuous if and only if $\pi_n\circ\fb: \Xb \to \Xb_n$ is continuous
for all $n$, where $\pi_n$ is the projection onto the $n$th coordinate.
\end{prop}

\begin{defi}
A topological space is {\it second countable} if it has a countable base
(i.e., $\Bb$ is countable). A subset $\Cb$ of a topological space is
{\it sequentially closed} if the limit of any convergent sequence in $\Cb$
belongs to $\Cb$. The {\it sequential closure} of a countable set $C$ in a
second countable space is the set of limit points of convergent sequences
in $C$.
\end{defi}

In the last part of this definition
we need $C$ to be countable and the ambient space to be
second countable so that the statement ``$x$ is the limit of a convergent
sequence in $C$'' satisfies excluded middle. This statement will be true
if and only if every basic open set that contains $x$ intersects $C$.

\begin{prop}
Any second countable space is separable.
\end{prop}

\begin{prop}
In any topological space, any closed set is sequentially closed. In a
second countable space, any separable sequentially closed set is closed.
\end{prop}

\begin{proof}
The first statement is trivial. For the second, let $\Cb$ be a separable
sequentially closed set with countable dense subset $C$; we construct the
complementary open set as the union of all basic open sets that do not
intersect $C$. Since $C$ is countable the condition $\Ub \cap C =
\emptyset$ satisfies excluded middle, so this union exists. Checking that
it is the complement of $\Cb$ is straightforward.
\end{proof}

\begin{coro}
In a second countable space the sequential closure of any countable set
is closed.
\end{coro}

\begin{defi}
A topological space is {\it compact} if the intersection of any sequence of
closed sets, any finitely many of which have nonempty intersection, is
nonempty. It is {\it sequentially compact} if every sequence has a
convergent subsequence.
\end{defi}

\begin{prop}
Any sequentially compact space is compact.
Any compact second countable space is sequentially compact.
\end{prop}

\begin{proof}
The first assertion is easy: given a sequence of closed sets $\Cb_n$
with the finite intersection property, and assuming sequential compactness,
for each $n$ choose $x_n \in \Cb_1 \cap \cdots \cap \Cb_n$ and then let $x$
be the limit of some convergent subsequence of $(x_n)$. It is easy to
see that $x$ must belong to the intersection $\bigcap \Cb_n$.

For the second assertion suppose $\Xb$ is compact and second countable
and let $(x_n)$ be a sequence in $\Xb$. Then for each $n$ the sequential
closure $\Cb_n$ of the set $\{x_k: k \geq n\}$ is closed by Corollary 3.54.
By compactness the intersection of these sets is nonempty, and any
point in this intersection is easily seen (using second countability)
to be the limit of some subsequence of $(x_n)$.
\end{proof}

\begin{theo}
Let $(\Xb_n)$ be a sequence of compact second countable spaces. Then
$\prod \Xb_n$ is compact and second countable.
\end{theo}

\begin{proof}
The fact that the product of a sequence of second countable spaces is
second countable follows easily from the definition of the product
topology. For compactness, use the equivalence of compactness and
sequential compactness and show that any sequence has a convergent
subsequence by successively extracting subsequences that converge on
the first $n$ coordinates and diagonalizing.
\end{proof}

\subsection{Measure theory}

Measure theory presents a greater challenge to formalization in CM
because its usual development involves uncountable pathology in
the form of, for example, the Borel hierarchy on the real line. However,
the fact that every measurable subset of $\Rb$ is a $G_\delta$ set minus
a null set strongly suggests that this kind of
pathology is not essential to the theory. Every measurable set is nested
between an $F_\sigma$ set and a $G_\delta$ set with null difference, which
motivates the following definition.

\begin{defi}
A {\it function on a family of subsets $\MM \subseteq \Xb \times \Mb$
of a set $\Xb$} is a function $\fb$ with domain $\Mb$ such that
$\MM_{(Y)} = \MM_{(Z)}$ implies $\fb(Y) \equiv \fb(Z)$. This allows us to
define {\it the value of $\fb$ on $\MM_{(Y)}$} to be $\fb(Y)$. We may
write $\fb(\MM_{(Y)})$ for $\fb(Y)$.

A {\it family of pairs of subsets} of a set $\Xb$ is a family of subsets
$\MM$ of $\Xb \times \{0,1\}$. For each $Y \in \Mb$
and $i = 0,1$ we write $\MM^i_{(Y)} = \{x \in \Xb: \langle x,i\rangle
\in \MM_{(Y)}\}$, and we also write $\MM_{(Y)} = \langle \MM^0_{(Y)},
\MM^1_{(Y)}\rangle$. $\MM$ is a {\it family of nested pairs of subsets}
if every pair $\langle \Yb_0, \Yb_1\rangle$ in $\MM$ satisfies
$\Yb_0 \subseteq \Yb_1$.

A {\it compatible function on a family of nested pairs of subsets of $\Xb$}
is a function $\mu$ on a family $\MM$ of nested pairs of subsets of $\Xb$
with the following property:
\begin{quote}
if $\langle\Yb_0, \Yb_1\rangle$ and
$\langle \Zb_0, \Zb_1\rangle$ belong to $\MM$ and
$\Yb_0 \cup \Zb_0 \subseteq \Yb_1 \cap \Zb_1$ then
$\mu(\langle\Yb_0, \Yb_1\rangle) = \mu(\langle \Zb_0, \Zb_1\rangle)$.
\end{quote}
Since $\Yb_0 \subseteq \Yb \subseteq \Yb_1$ and
$\Zb_0 \subseteq \Yb \subseteq \Zb_1$ imply $\Yb_0 \cup \Zb_0
\subseteq \Yb_1 \cap \Zb_1$, the compatibility condition allows us to
define {\it the value of $\mu$ on $\Yb$} to be $\mu(\langle \Yb_0,
\Yb_1\rangle)$ for any subset $\Yb \subseteq \Xb$ such that
$\Yb_0 \subseteq \Yb \subseteq \Yb_1$. We may write $\mu(\Yb)$ for
$\mu(\langle \Yb_0,\Yb_1\rangle)$.
We say that such a set $\Yb$ is {\it measurable} or {\it $\mu$-measurable}.

A {\it measure} on a set $\Xb$ is a compatible function
$\mu: \Mb \to [0,\infty]$ on a family of nested pairs of subsets
of $\Xb$ such that

(i) $\emptyset$ is measurable and $\mu(\emptyset) = 0$;

(ii) if $\Yb$ is measurable then so is $\Xb - \Yb$;

(iii) if each set in a sequence $(\Yb_n)$ is measurable then
so is their union, and if the sets are disjoint then
$\mu(\bigcup \Yb_n) = \sum \mu(\Yb_n)$.
\end{defi}

The problem of constructing measures also requires a new technique.
We cannot use
Carath\'{e}odory's method because it defines measurability using
what would be in our context a third order quantification. However,
it is not hard to come up with a more direct construction that also
works. We consider only the case of finite measures, but passing to
$\sigma$-finite measures would be a simple matter of partitioning into
finite measure subspaces.

\begin{theo}
Let $\Xb$ be a set, let $\widetilde{\MM} \subseteq \Xb \times \widetilde{\Mb}$
be a nonempty family of subsets of $\Xb$ which is stable under finite unions
and complements, and let $\tilde{\mu}: \widetilde{\Mb} \to [0,a]$ be a
function on the family $\widetilde{\MM}$. Suppose that
$\tilde{\mu}(\emptyset) = 0$ and $\tilde{\mu}(\bigcup \Xb_n) =
\sum \tilde{\mu}(\Xb_n)$ whenever $(\Xb_n)$ is a disjoint sequence of sets
that belong to the family whose union also belongs to the
family. Then there is a measure $\mu$ on $\Xb$ such that every set that belongs
to the family $\widetilde{\MM}$ is measurable and $\mu$ agrees with $\tilde{\mu}$
on every such set.
\end{theo}

\begin{proof}
We merely indicate the construction of $\mu$. The verification that $\mu$
has the desired properties is an exercise in measure
theory and we omit it.

For any $Y$ and $Z$ in $\widetilde{\Mb}$ define $\db(Y,Z) =
\tilde{\mu}(\widetilde{\MM}_{(Y)}\Delta\widetilde{\MM}_{(Z)})$,
where $\Delta$ denotes symmetric difference. This is a pseudometric
on $\widetilde{\Mb}$. Then let $\Mb$ be the set of all $Y$ such that
$Y_{(n)}$ belongs to $\widetilde{\Mb}$ for all $n$ and
$\db(Y_{(m)},Y_{(n)}) \to 0$ as $m,n \to \infty$. We define
$\MM$ by the prescription $\MM^0_{(Y)} = \liminf \widetilde{\MM}_{(Y_{(n)})}$
and $\MM^1_{(Y)} = \limsup \widetilde{\MM}_{(Y_{(n)})}$, and we set
$\mu(\langle \MM^0_{(Y)}, \MM^1_{(Y)}\rangle) =
\lim \tilde{\mu}(\widetilde{\MM}_{(Y_{(n)})})$. This completes the
construction of $\mu$.
\end{proof}

\begin{defi}
The function $\tilde{\mu}$ in Theorem 3.59 is a {\it premeasure}, and
$\mu$ is the measure {\it generated} by $\tilde{\mu}$. A measure is
{\it separable} if it is generated by a premeasure defined on a
countable family of subsets.
\end{defi}

Theorem 3.59 allows us to construct Lebesgue measure in $[0,1]^n$ in
the usual way, or in $\Rb^n$ by partitioning into cubes.

Integration can be defined using similar methods. The definition is framed
in terms of a generating premeasure but it is not hard to see that the
integral does not actually depend on the choice of premeasure.

\begin{defi}
Let $\mu$ be a measure generated by a premeasure $\tilde{\mu}$. We
say that a function $\fb: \Xb \to \Rb$ is {\it simple} if it is a
finite linear combination of characteristic functions of sets that belong
to the family $\widetilde{\MM}$. We define the {\it integral} of a simple
function $\fb = \sum a_i \chi_{\Ab_i}$ to be
$$\int \fb \, d\tilde{\mu} = \sum a_i \tilde{\mu}(\Ab_i)$$
and we define the {\it $L^1$ distance} between two simple functions $\fb$
and $\gb$ to be
$$\db(\fb,\gb) = \int |\fb - \gb|\, d\tilde{\mu}.$$
A function $\fb: \Xb \to \Rb$ is {\it integrable} if there is a
sequence $(\fb_n)$ of simple functions, Cauchy for $L^1$ distance, such that
$$\liminf \fb_n \leq \fb \leq \limsup \fb_n.$$
We then define its {\it integral} $\int \fb\, d\mu$ to be
$$\int \fb\, d\mu = \lim \int \fb_n\, d\tilde{\mu}.$$
\end{defi}

\begin{theo}
The integral $\int \fb\, d\mu$ is well-defined.
\end{theo}

Finally, we indicate how to get a version of the Radon-Nikodym theorem.
The technique of sequential approximation is again crucial.

\begin{defi}
A {\it signed measure} on $\Xb$ is a compatible function $\nu: \Mb \to \Rb$
on a family of
nested pairs of subsets of $\Xb$ that satisfies the same axioms as a
measure. A signed measure $\nu$ is {\it absolutely continuous} with
respect to a measure $\mu$ if every $\mu$-measurable set is $\nu$-measurable
and $\mu(\Yb) = 0$ implies $\nu(\Yb) = 0$.
\end{defi}

\begin{theo}
Let $\mu$ be a separable finite measure on $\Xb$ and let $\nu$ be a (finite)
signed measure on $\Xb$ that is absolutely continuous with respect to $\mu$.
Then there is a $\mu$-integrable function $\fb: \Xb \to \Rb$ such that
$$\nu(\Ab) = \int \fb\cdot\chi_\Ab\, d\mu$$
for every $\mu$-measurable set $\Ab$.
\end{theo}

\begin{proof}
Again we merely indicate the construction. First, by separability there
is a generating premeasure defined on a countable algebra of sets.
We can then find a sequence of finite partitions of $\Xb$ by sets in
the algebra, such that every finite partition of $\Xb$ by sets in the
algebra is refined by some member of the sequence. If the $n$th
partition is $\Xb = \Ab_1 \cup \cdots \cup \Ab_k$ then we define
$$\fb_n = \sum_{j=1}^k \nu(\Ab_j)\chi_{\Ab_j}.$$
We then check that the sequence $(\fb_n)$ is Cauchy and that this
implies that it converges absolutely on a set of full measure to a
$\mu$-integrable function $\fb$. This completes the construction of $\fb$.
\end{proof}

\subsection{Banach spaces}

For simplicity we take the scalar field to be real; complex scalars do not
carry any additional logical demands.

Our definition of Banach spaces is identical to the classical one. What is
noteworthy here is that most of the classical examples require some sort of
coding. But little $L^p$ spaces do not:

\begin{defi}
For $1 \leq p < \infty$ let $l^p$ be the set of all sequences $(a_n)$
of real numbers such that $\sum |a_n|^p < \infty$, with norm
$\|(a_n)\|_p = \big(\sum |a_n|^p\big)^{1/p}$. Let $l^\infty$ be the
set of all bounded sequences of real numbers, with norm
$\|(a_n)\|_\infty = \sup |a_n|$.
\end{defi}

Here the condition $\sum |a_n|^p < \infty$ satisfies excluded middle
because it is equivalent to the condition ``there exists $K > 0$ such
that $\sum_{n=1}^m |a_n|^p \leq K$ for all $m$''. The norm itself exists
because the sequence of partial sums can be constructed using dependent
choice, and the supremum of that sequence can then be taken by Cauchy
completeness of $\Rb$.

Spaces of the form $C(\Xb)$ with $\Xb$ a compact metric space cannot be
directly represented in CM because each element is supposed to be a
third order object (a function from $\Xb$ into $\Rb$). However, this is
not a serious problem because any continuous function is determined by
its values on a dense subset.

\begin{defi}
Let $\Xb$ be a separable compact metric space with countable dense
subset $X$. We define $C(\Xb)$ to be the set of all uniformly continuous
functions from $X$ to $\Rb$.
\end{defi}

Literally, $C(\Xb)$ is the set of bounded sequences $(a_n) \in l^\infty$ such
that the map $x_n \mapsto a_n$ from $X$ to $\Rb$ is uniformly continuous, for
some enumeration $(x_n)$ of $X$. $C(\Xb)$ inherits its Banach space
structure from $l^\infty$.

At the Banach space level there is no particular advantage to working
with the functions themselves rather than their restrictions to a dense
subset. However, we certainly want to be able to work with individual
elements of $C(\Xb)$ as continuous functions on $\Xb$. This is easily
seen to be possible:

\begin{prop}
Let $\Xb$ be a separable compact metric space with countable dense subset
$X$. Then the restriction of any continuous function $\fb: \Xb \to \Rb$
to $X$ defines a sequence $(a_n)$ in $C(\Xb)$, and every sequence
$(a_n)$ in $C(\Xb)$ is the restriction of precisely one continuous
function.
\end{prop}

We use the Radon-Nikodym theorem (Theorem 3.64) to encode $L^p$
functions:

\begin{defi}
Let $\Xb$ be a separable finite measure space with generating premeasure
$\tilde{\mu}$. We define $L^1(\Xb)$ to be the set of all signed premeasures
$\tilde{\nu}$ on the family $\widetilde{\MM}$ which are absolutely continuous
with resepect to $\tilde{\mu}$, i.e., for all $\epsilon > 0$ there exists
$\delta > 0$ such that
$$\sum \tilde{\mu}(\Ab_i) \leq \epsilon\qquad \Rightarrow\qquad
\sum |\tilde{\nu}(\Ab_i)|\leq \delta$$
for any disjoint $\Ab_1, \ldots, \Ab_n$ in the algebra.
$L^\infty(\Xb)$ consists of the premeasures which satisfy the
stronger condition that there exists $K \geq 0$ such that
$$\sum |\tilde{\nu}(\Ab_i)| \leq K\cdot \sum \tilde{\mu}(\Ab_i)$$
for any disjoint $\Ab_1, \ldots, \Ab_n$ in the algebra.
We define $L^p(\Xb)$ for $1 < p < \infty$
to be those premeasures in $L^1(\Xb)$ the $p$th power of whose
Radon-Nikodym derivative is bounded.
\end{defi}

As for $C(\Xb)$, elements of $L^p(\Xb)$ are literally sequences of real
numbers which become premeasures when composed with a bijection
from $\widetilde{\Mb}$ to $N$. Again, the above definition could be
extended to the $\sigma$-finite case by partitioning into finite measure
subsets.

The following analog of Proposition 3.67 is an immediate consequence of
Theorem 3.64.

\begin{prop}
For $1 \leq p < \infty$ the elements of $L^p(\Xb)$ correspond to
functions on $\Xb$, modulo alteration on a set of measure zero, the
$p$th power of whose absolute value is integrable. The elements of
$L^\infty(\Xb)$ correspond to bounded integrable functions on $\Xb$,
modulo alteration on a set of measure zero.
\end{prop}

In the $\sigma$-finite case we no longer have
$L^p(\Xb) \subseteq L^1(\Xb)$, but we still have $L^p(\Xb)
\subseteq L^1_{loc}(\Xb)$, so can adapt the above result to this case.

Next we discuss duality.

\begin{defi}
Let $\Eb$ be a separable Banach space with countable dense subset
$E$. We may assume that $E$ is a vector space over $Q$ (cf.\ Lemma
3.11). We define the {\it dual Banach space} $\Eb'$ to be the set
of bounded $Q$-linear maps from $E$ to $\Rb$. The norm on $\Eb'$
is defined by $\|f\| = \sup\{|f(x)|: x \in E, \|x\| \leq 1\}$.
\end{defi}

As before, the elements of $\Eb'$ are modelled as sequences of real numbers.

\begin{prop}
The restriction of any bounded linear functional on $\Eb$ to $E$
defines an element of $\Eb'$, and every element of $\Eb'$ is the
restriction of precisely one bounded linear functional on $\Eb$.
\end{prop}

We can prove a version of the Hahn-Banach theorem:

\begin{theo}
Let $\Eb$ be a separable Banach space, let $\Eb_0$ be a separable
closed subspace, and let $\fb_0: \Eb_0 \to \Rb$ be a bounded
linear functional on $\Eb_0$. Then $\fb_0$ extends to a bounded linear
functional $\fb$ on $\Eb$ with $\|\fb\| = \|\fb_0\|$.
\end{theo}

\begin{proof}
Since $\Eb$ is separable, we can enumerate a dense subset $(x_n)$,
and it will suffice to show that $\fb_0$ extends to $\Eb_0 + \Rb\cdot x_1$;
we can then recursively extend to the span of $\Eb_0$ and
$x_1, \ldots, x_n$, use dependent choice to extract a nested
sequence of extensions, and amalgamate them.

The extension to $\Eb_0 + \Rb x_1$ is effected just as in the
classical proof. We need $\Eb_0$ to be separable so that the
classical inequality
$$\sup_{x \in \Eb_0} (-\|\fb_0\| \|x_1 + x\| - \fb_0(x))
\leq \fb(x_1) \leq \inf_{x \in \Eb_0} (\|\fb_0\| \|x_1 + x\| - \fb_0(x))$$
can be restricted to $x$ ranging over a countable dense subset of
$\Eb_0$, in order to ensure that the supremum and infimum exist.
\end{proof}

The same result holds, with the same proof, for extensions from
separable subspaces of nonseparable spaces, but this requires the
well-ordering $\prec$ of CM${}^+$. (See Section 2.3.)

The weak* topology on the dual of a separable Banach space $\Eb$ is defined
in the usual way. Note that its restriction to the unit ball of $\Eb'$ is
second countable, and even metrizable.

\begin{theo}
The closed unit ball of the dual of any separable Banach space is
weak* compact.
\end{theo}

\begin{proof}
We verify sequential compactness. This is enough by Proposition 3.56.
To do this let $E$ be a countable dense $Q$-linear subspace of a
separable Banach space $\Eb$ and let $(f_n)$ be a sequence of
bounded linear functionals on $E$, each of norm at most 1. Enumerating
$E$ as $(x_n)$, we then successively extract subsequences of $(f_n)$
which converge on $x_1, \ldots, x_k$. Diagonalizing yields a subsequence
$(f_{n_k})$ such that the sequence $(f_{n_k}(x_i))$ converges for every $i$.
Thus every sequence has a weak* convergent subsequence. (It suffices to verify
convergence on a dense set in $\Eb$ since the sequence $(f_n)$ is bounded.)
\end{proof}

We close with a version of Goldstine's theorem. This is interesting because
it is a basic theorem about the second dual, yet in general second duals,
even of separable Banach spaces, cannot be constructed in CM. Separability
of $\Eb$ does not imply separability of $\Eb'$, but we need $\Eb'$ to be
separable in order to construct $\Eb''$.

Our inability to form second duals might appear
to reveal a serious limitation in our ability to formalize
standard functional analysis within CM. But the limitation is not severe
because typical applications of $\Eb''$ do not involve its Banach space
structure. Rather, they have to do with the behavior of individual elements
of $\Eb''$, which are not excluded from CM. (Though as we mentioned just
above, if $\Eb'$ is nonseparable we would need to work in CM${}^+$ to prove
the existence of elements of $\Eb'' - \Eb$.)

Goldstine's theorem is a good illustration of this phenomenon. Its classical
statement is that the unit ball of $\Eb$ is weak* dense in the unit ball of
$\Eb''$. But this really comes down to an assertion about weak* approximability
of individual elements of $\Eb''$ by elements of $\Eb$. That version of the
result can be stated and proven in CM.

\begin{theo}
Let $\Eb$ be a separable Banach space and let $\phi: \Eb' \to \Rb$ be a bounded
linear functional of norm at most 1. Then for any $\fb_1, \ldots, \fb_n$ in
$\Eb'$ and any $\epsilon > 0$ there exists $x$ in the unit ball of $\Eb$
such that
$$|\phi(\fb_i) - \fb_i(x)| < \epsilon$$
for $1 \leq i \leq n$.
\end{theo}

\begin{proof}
In the statement of the theorem we have identified the linear functionals
$\fb_i$ on $\Eb$ with their representatives $f_i$ in $\Eb'$. That is,
$f_i$ is the restriction of $\fb_i$ to a countable dense $Q$-linear subspace
$E$ of $\Eb$. Now fix $\fb_1, \ldots, \fb_n$ and $\epsilon$; we claim that
there exists $x \in E$ with the desired properties. Since $E$ is countable
this assertion satisfies excluded middle, so we can prove it by contradiction.

Thus, suppose no $x \in E$ satisfies $\|x\| \leq 1$ and
$|\phi(\fb_i) - f_i(x)| < \epsilon$ for all $i$. Consider the map $T: E \to
\Rb^n$ defined by $T(x) = (f_1(x), \ldots, f_n(x))$. Then the closure
$K = \overline{T([E]_1)}$ is a convex subset of $\Rb^n$ (here $[E]_1$
denotes the unit ball of $E$) and it is separated from the point
$\alpha = (\phi(\fb_1), \ldots, \phi(\fb_n))$ by a distance of at least $\epsilon$.
So by a separation theorem for separable convex subsets of $\Rb^n$, which
has easy elementary proofs, we can find a linear map $\gb: \Rb^n \to \Rb$
such that $\gb(\beta) \leq 1 < \gb(\alpha)$ for all
$\beta \in K$. Finally, the map $\gb \circ T$ belongs to the unit ball of
$\Eb'$ but we have $\phi(\gb \circ T) = \gb(\alpha) > 1$ by linearity, which
contradicts the assumption that $\phi$ has norm at most 1. This shows that
the desired $x$ does exist.
\end{proof}



\begin{thebibliography}{aaaaaaaa}
\bibitem {Par}
C.\ Parsons, What is the iterative conception of set?, in \cite{BP},
503-529.
\end{thebibliography}
\end{document}